\documentclass{amsart}
\usepackage[dvips]{color}
\usepackage{graphicx}
\usepackage{epsfig}
\usepackage{tikz}
\usepackage{enumerate}
\usepackage{color}

\usepackage{latexsym}
\usepackage{amssymb}
\usepackage{amsmath}
\usepackage{amsthm}
\usepackage{amscd}

\theoremstyle{plain}
\newtheorem{thm}{Theorem}[section]
\newtheorem*{thm*}{Theorem}
\newtheorem{lem}[thm]{Lemma}
\newtheorem{defin}[thm]{Definition}
\newtheorem{prop}[thm]{Proposition}
\newtheorem*{prop*}{Proposition}

\theoremstyle{definition}

\newtheorem{obs}[thm]{Observation}

\newtheorem{example}[thm]{Example}
\newtheorem{ques}{Question}
\newtheorem{cor}[thm]{Corollary}

\newcommand{\NN}{\mathbb{N}}

\newcommand{\RR}{\mathbb{R}}

\begin{document}
\title{The set of uniquely ergodic IETs is path-connected}
\author[J.~Chaika]{Jon Chaika}
\thanks{J. Chaika was supported in part by NSF grant DMS 1330550.}
\address{University of Utah, Department of Mathematics, 209\\
155 S 1400 E RM 233\\
Salt Lake City, UT, 84112-0090 USA
}
           \email{chaika@math.utah.edu}

\author[S.~Hensel]{Sebastian Hensel}
\address{The University of Chicago, Department of Mathematics\\
5734 South University Avenue, Chicago, Illinois 60637 USA}
           \email{hensel@math.uchicago.edu}

\maketitle
\begin{figure}[h]
  \centering
  \includegraphics[width=0.6\textwidth]{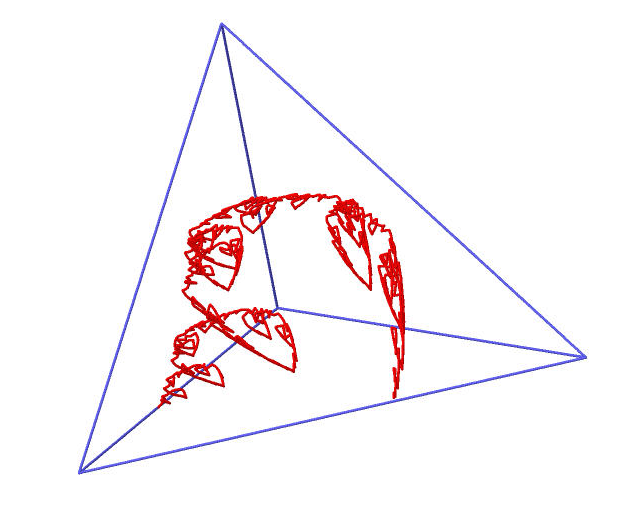}
  \caption{A path of uniquely ergodic $4$-IETs}
\end{figure}
\section{Introduction}\label{sec:intro}
An interval exchange transformation (from now on
abbreviated to \emph{IET}) is a piece-wise isometric map of an
interval to itself that rearranges sub-intervals according to a 
permutation $\pi$ (see Section~\ref{sec:iets} for formal definitions).
While simple to define, interval exchange transformations have deep and interesting dynamical
properties as well as connections to foliations on Riemann surfaces, and
therefore have been the subject of intense study over the last
years. See \cite{zorich}, \cite{yoccoz} or \cite{viana survey} for good surveys.

In this article we are concerned not with the dynamics of a single
IET, but the set of all IETs with a given permutation $\pi$ of $n$
symbols. This set 
carries a natural topology and is homeomorphic to an Euclidean simplex
$\Delta_\pi$ of dimension $n-1$.
It is known that unless the permutation has obvious combinatorial
obstructions most IETs are uniquely ergodic:
\begin{thm*}(Masur \cite{masur}, Veech \cite{gauss measures}) 
Given an irreducible permutation almost every IET with that
permutation is uniquely ergodic with respect to Lebesgue measure. 
\end{thm*}
On the other hand, it is known that the set of non-uniquely ergodic
IETs is also fairly large in a geometric sense:
\begin{thm*}(Masur-Smillie \cite{masur smillie}) 
The set of minimal and not uniquely ergodic IETs with a non-degenerate
permutation on $n$ intervals has Hausdorff dimension greater than
$n-1$. 
\end{thm*}
The Hausdorff dimension of minimal and not uniquely ergodic 4-IETs has been computed by the J. Athreya and the first named author.
\begin{thm*}[Athreya-Chaika \cite{athreya-chaika}] The set of minimal and not uniquely ergodic IETs with a non-degenerate
permutation on $4$ intervals has Hausdorff dimension greater than
$\frac 5 2$. 
\end{thm*}
In this article we investigate the set of uniquely ergodic IETs from a
topological point of view. More precisely, the main result of this paper is 
\begin{thm}\label{thm:main}
  Let $n\geq 4$ and let $\pi$ be any non-degenerate permutation on $n$
  symbols. Then the set of uniquely ergodic unit length IETs with
  permutation $\pi$ is path connected.
\end{thm}
For the formal definition of non-degenerate permutations see
Section~\ref{sec:iets}. Intuitively, non-degenerate means that no
induced map on a sub-interval has less than $n$ singularities.  In the
case of $n=2$, the space of $2$-IETs is equal to the space of
rotations; where such an IET is uniquely ergodic if and only if the rotation
angle is irrational, which is clearly a disconnected subset.

$3$-IETs can be thought of as the induced map (see Definition
\ref{defin:first ret}) of a rotation by ${\frac{1-L_1}{1+L_2}}$ on
$[0,\frac 1 {1+L_2})$ (compare \cite{KS}, \cite{keane}) and they are
uniquely ergodic if and only if $\frac{1-L_1}{1+L_2}$ is irrational.
So a set $\frac{1-L_1}{1+L_2}=y \in \mathbb{Q}$ gives a curve (or
plane if one does not normalize the lengths of an IET) of non-minimal
(and so not uniquely ergodic) IETs that disconnect the space of
$3$-IETs.

Thus the bound on $n$ in Theorem~\ref{thm:main} is optimal.

\subsection*{Outline of proof}
The proof of Theorem~\ref{thm:main} begins with a reduction to the
case of $n=4$. Namely, the space of $n$-IETs contains many copies 
of the space of $4$-IETs by setting the length of enough intervals to
be $0$ --  and in these subspaces we know that the uniquely ergodic
IETs are path-connected. Studying the combinatorics of a
non-degenerate permutation, we can show that the union $U$ of all these
subspaces is path-connected. Using a limiting argument we can then
show that each uniquely ergodic IET can be connected by a path to a
point in $U$.

\smallskip In the case of $n=4$ we are able to give an explicit
inductive procedure that constructs paths in the space of
uniquely ergodic IETs. Unique ergodicity of an IET $T$ can be detected by
the Rauzy induction of $T$. Thus ideally one would want to construct
paths $c$ where Rauzy induction has desirable properties for each $c(t)$.

A key problem however is that Rauzy induction is undefined on a
countable set of codimension $1$ planes in the simplex. Continuous
paths necessarily cross these ``fail planes'', and in general most
points on such a plane do not correspond to uniquely ergodic IETs.

This suggests the following naive strategy: start by joining two IETs
$S,T$ by a straight line segment $S\to T$. If this line intersects the
plane where Rauzy induction fails, select a uniquely ergodic point $R$
on that plane, and replace the line segment by a concatenation
$S\to R\to T$ of two segments. Now both $S\to R$ and $R\to T$ share
the same first Rauzy induction step.  We can continue this process
iteratively, always replacing straight segments by concatenations
after some number of Rauzy steps.

There are two main issues with this approach: why does the sequence of
approximate paths converge, and why are all points on it uniquely
ergodic?

We deal with both of these issues simultaneously by taking care
how to choose the intermediate IETs on the fail planes. More precisely
we want that limit points which do not eventually follow the Rauzy
expansion of one of the points on a fail plane to have infinitely many matrices
which define contracting maps on the simplex in their Rauzy
expansion. This will yield both
unique ergodicity and continuity at these limit points. For the
points on the fail planes, unique ergodicity is clear from the
construction, but continuity needs to be proved by a different
argument. 

\medskip Section~\ref{sec:iets} of the article describes some
necessary background on IETs and Rauzy induction, while
Section~\ref{sec:matrices} contains some basic results and notation on
matrices that is used throughout.

Section~\ref{sec:abstract} describes the basic mechanism used to
construct the paths in the case $n=4$ (depending on an explicit
construction done in Section~\ref{sec:concrete}) -- the necessary
results to show continuity at points on ``fail planes'' is developed
in Section~\ref{sec:finite-depth}, the results for other points is
developed in Section~\ref{sec:infinite-depth}. Finally,
Section~\ref{sec:connecting-building-blocks} contains the proof of
Theorem~\ref{thm:main} for $n=4$.  In Section~\ref{sec:n-iets} we
extends the results to $n$-IETs. 

\subsection*{Questions}
The work in this article suggests several possible directions of
further research. One possibility is to ask further topological
question about sets of specific IETs. In particular, we have
\begin{ques}[Mladen Bestvina]
  In the case of $n$-IETs for $n\geq 5$, does the set of uniquely
  ergodic $n$-IETs satisfy higher connectivity properties?
\end{ques}
\begin{ques}
  Is the set of minimal $n$-IETs path connected?
\end{ques}
On the other hand, one could try to generalize the question of the
topology of the set of uniquely ergodic transformations from IETs to
other settings. One example is the following
\begin{ques} 
  Is there 1-parameter diagonal flow on $SL_4(\mathbb{R})/SL_4(\mathbb{Z})$
  so that the set of points whose orbit under this flow is bounded path
  connected?
\end{ques}
Another example is a question that inspired this work.
The space of measured foliations on a Riemann surface of genus $g$ is
homeomorphic to a sphere of dimension $6g-7$. The set of uniquely
ergodic laminations has full measure \cite{masur}, and one can ask
\begin{ques}\label{q:lams}
  For a surface of genus $g\geq 2$, is the set of uniquely ergodic
  laminations path-connected? 
\end{ques}
There is a closely related space $\mathcal{EL}$ (the space of ending
laminations) which has been studied extensively. Ending lamination
space is known to be connected and locally path-connected \cite{gab},
\cite{leinshleim}, and even has been identified explicitly in some
low-complexity cases \cite{henspryz}, \cite{gabai2}. However, the set
of uniquely ergodic foliations is only a subspace of ending lamination
space, and so these results do not yield any direct 
information about the topology of the set of uniquely ergodic
laminations. Somewhat in the line of our results is a result by
Leininger-Schleimer which ensures the existence of spheres in the set
of uniquely ergodic laminations \cite{leinshleim2}. To the authors'
knowledge Question~\ref{q:lams} is still open.

\smallskip
Once could also try to adapt the methods of the current paper to the
setting of surface foliations by building flat surfaces from IETs. 
The direct analog of the case $n=4$ (which all our work is inductively
based on) would concern foliations in the stratum
$\mathcal{H}(2)$ in genus $2$. However, in this stratum
every saddle connection defines a simple closed curve, and thus 
arational foliations are not path connected. Indeed in every Rauzy class there are 
permutations where the IETs corresponding to arational foliations are not path connected.
 However, one can still ask the following question:
\begin{ques}
For a suitable permutation $\pi$, is the set of all IETs
corresponding to arational uniquely ergodic foliations path-connected?
\end{ques}

\textbf{Acknowledgments}: We would like to thank Jayadev Athreya and
Howard Masur for inspiring discussions and interest in the
project. Furthermore, we express our gratitude for the numerous and
detailed comments by the anonymous referee, which helped to improve
the paper greatly.

\section{Interval Exchange Transformations and Rauzy Induction}
\label{sec:iets}
In this section we fix the notation for interval exchange
transformations (\emph{IETs}) that we will use, and also collect some
well known results that we will use throughout. We refer the reader to 
\cite{viana survey} for a detailed treatment.

\medskip
We denote a permutation $\pi$ on the symbols $1,\ldots, n$ by the list
containing the image of the 
symbols under the inverse permutation $\pi^{-1}$. That is, the
identity permutation on  
four symbols is denoted by $(1234)$, the transposition of the 
first two symbols is $(2134)$. The permutation which maps every symbol
to the next one (cyclically) would be $(4123)$. The reason for this
convention will be clear later.

\begin{defin} Given $\hat{L}=(l_1,l_2,...,l_d)$
where $l_i \geq 0$, we obtain $d$ sub-intervals of the
interval $\left[0,\underset{i=1}{\overset{d}{\sum}} l_i\right):$ $$I_1=[0,l_1) ,
I_2=[l_1,l_1+l_2),...,I_d=[l_1+...+l_{d-1},l_1+...+l_{d-1}+l_d).$$ Given
 a permutation $\pi$ on  the set $\{1,2,...,d\}$, we obtain a
 d-\emph{Interval Exchange Transformation} (IET)  
 $$T \colon
 \left[0,\underset{i=1}{\overset{d}{\sum}} l_i\right) \to 
 \left[0,\underset{i=1}{\overset{d}{\sum}} l_i\right)$$ 
 which exchanges the
 intervals $I_i$ according to $\pi$. That is, if $x \in I_j$
 then $$T(x)= x - \underset{k<j}{\sum} l_k
 +\underset{\pi(k')<\pi(j)}{\sum} l_{k'}.$$ 

 If such a $T$ is given, we denote its (unnormalized) length vector by
 $\hat{L}(T)$ and its permutation by $\pi(T)$. 
\end{defin}
With our convention, $\pi(T)$ is then described by the order of
intervals from left to right after applying the IET.  
If $T$ is an IET, we denote by
\[L(T) = \hat{L}(T)/|\hat{L}(T)|_1\]
the (normalized) length vector of $T$. We say that a IET is
\emph{normalized} if it is defined on the unit interval. Every IET can
be rescaled to a normalized one, and we will usually identify IETs
which just differ by such a rescaling without explicit mention.  When
we speak about a length vector we mean the normalized one, unless
specified explicitly.

Intuitively, a permutation on $n$ symbols is degenerate if any IET
with that permutation has either fewer than $n-1$ discontinuities or
an induced map of (see below for the definition) it has fewer than
$n-1$ discontinuities. The technical conditions are as follows (\cite[Section
3]{iet v})
\begin{defin}\label{def:degenerate}
  A permutation $\pi$ on $n$ symbols is \emph{degenerate} if there is some
  $j<n$ so that one of the following holds.
  \begin{enumerate}
  \item $\pi(j+1)=\pi(j)+1$
  \item $\pi(j)=n, \, \pi(j+1)=1$ and $\pi(1)=\pi(n)+1$
  \item $\pi(j+1)=1$ and $\pi(1)=\pi(j)+1$
  \item $\pi(j+1)=\pi(n)+1$ and $\pi(j)=n$
  \end{enumerate}
\end{defin}

\begin{defin}\label{defin:first ret} Let $T:[0,1) \to [0,1)$ be measurable and Lebesgue measure preserving and $A \subset [0,1)$. The \emph{induced map of $T$ on $A$} is 
$T_A:A \to A$ given by  
\[ T_A(x)=T^{n(x)}(x), \quad\quad n(x) = \min\{n>0:T^n(x)\in A\}.\] 
\end{defin}
Recall that the Poincar\'e recurrence theorem guarantees that $T_A$ is
defined almost everywhere in $A$.

Next, we briefly describe the most important points of \emph{Rauzy
  induction}, the renormalization method we will use throughout to
study IETs. Our treatment of Rauzy induction will be the same as in
\cite[Section 7]{gauss measures}.

Let $T$ be a $n$-IET with permutation $\pi$. Let $\delta_+$ be the
rightmost discontinuity of $T$ and $\delta_-$ be the rightmost
discontinuity of $T^{-1}$. Let
$\delta_{max}=\max\{\delta_+,\delta_-\}$. Consider the induced map of
$T$ on $[0,\delta_{\max})$ denoted $T|_{[0,\delta_{\max})}$. The result is
again an IET, perhaps with a different permutation. 
We can renormalize it so that it is once again a $n$-IET on
$[0,1)$. That is, let $R(T)(x)= \frac{1}{\delta_{\max}}T|_{[0,\delta_{\max})}(x\delta_{\max})$. This is the Rauzy induction of $T$.

If $\delta_+ \neq \delta_-$ then the permutation and length vector of the IET $R(T)$ can be combinatorially determined. 

Namely, if $\delta_{max}= \delta_+$ we say the first step in Rauzy induction is
$a$. In this case the permutation of $R(T)$ is given by
\begin{equation*} \pi'(j)= \begin{cases}
    \pi (j) & \quad j \leq \pi^{-1}(n)\\ \pi(n) & \quad j=\pi^{-1}(n)+1 \\ \pi(j-1) & \quad \text{otherwise}
  \end{cases}.
\end{equation*}
We keep track of what has happened to the interval lengths under Rauzy induction by a matrix
$M(T,1)$ where  
\begin{equation*} M(T,1)[ij]= \begin{cases} \delta_{i,j} & \quad j \leq \pi^{-1}(n)\\
 \delta_{i, j-1} & \quad j>\pi^{-1}(n) \text{ and } i \neq n\\
\delta_{\pi^{-1}(n)+1,j} & \quad i=n \end{cases}. 
\end{equation*}
If $\delta_{max}= \delta_-$ we say the first step in Rauzy induction is $b$.
In this case the permutation of $R(T)$ is given by 
\begin{equation*} \pi'(j)= \begin{cases}
 \pi (j) & \quad \pi(j) \leq \pi(n)\\ \pi(j)+1 & \quad \pi(n) < \pi(j) < n \\ \pi(n)+1 & \quad \pi (j)=n
\end{cases}.
\end{equation*}
Again, we keep track of what has happened to the interval lengths under Rauzy induction by a matrix \begin{equation*}M(T,1)[ij]= \begin{cases} 1 & \quad i=n \text{ and }j= \pi^{-1}(n) \\ \delta_{i,j} & \quad \text{ otherwise} \end{cases}. 
\end{equation*}
In the case where $\delta_+ = \delta_-$ both the formulas in case a
and case b allow to identify $R(T)$ with an IET. We write $R_a(T)$ and
$R_b(T)$ for these choices.

The change in permutation under Rauzy induction can be depicted in a \emph{Rauzy
  diagram}. Below is the diagram for $n=4$. 
\begin{figure}[h]\label{fig:rauzy-diagram}
\begin{tikzpicture}
\path[dotted,->]  (-.6,0) edge (-2.5, 1.7)
node[left=1.4 cm, above=1.4 cm, text width= 3cm]{(2431)};
\path[dotted,->] (-2.9, 1.4) edge (-2.9,.4);
\path[dotted,->] (-2.3,0) edge (-.9,0)
node[right=.35 cm, above=-.2 cm, text width=3cm]{(3241)};
\path[->] (-3.5,.17) edge [loop left] (-4,.15);
\path[->](-3.5,1.6) edge[bend left] (-5,1.6);
\path[->](-5,1.7) edge[bend left] (-3.5,1.7)
node[left= -.35 cm, above=-.35 cm, text width=3cm]{(2413)};
\path[dotted,->] (-6.2,1.6) edge[loop left](-6.2,1.6);
\path[->] (2,0) edge(.2,0);
\path[->](2.5,1.7) edge (2.5,.5)
(2.5,.5) node[ right= -1.9 cm, above=-.7 cm, text width=3 cm] {(4321)}
node[ right=1.2 cm, above=-.7 cm, text width=3cm]{(4213)}
node[ right=1.2 cm, above= 1.1 cm, text width=3cm]{(4132)};
\path[dotted,->] (3.2,.1) edge[loop right] (2.3,-.7);
\path[->](.2,.2) edge (2.2,1.9);
\path[dotted, ->] (3.2,2.0) edge[bend left] (4.5,2.0) 
node[ right=3 cm, above=-.3 cm, text width=3cm]{(3142)};
\path[dotted,->]  (4.5,1.9) edge[ bend left] (3.2,1.9);
\path[->](5.7,2.0) edge[loop right] (6,2);
\end{tikzpicture}
\caption{The Rauzy diagram for $n=4$.}
\end{figure}
 
The matrices described above depend on whether the step is $a$ or $b$
and the permutation $\pi(T)$. The following well known lemmas which are
immediate calculations help motivate the definition of $M(T,1)$. 
\begin{lem} \label{one step} If $R(T)=S$  then the length
  vector $L(T)$ is a scalar multiple of $M(T,1)L(S)$. 
\end{lem} 
We will denote the positive orthant by
\[ \RR^d_+ = \left\{ (x_1,\ldots,x_d)\in\RR^d | x_i>0 \,\mbox{for all}\, i\right\} \]
and the open Euclidean simplex by
\[ \mathring{\Delta}_d = \left\{ (x_1,\ldots,x_d)\in\RR_+^d, \sum_{i=1}^d x_i = 1  \right\}. \]
We let $M{\Delta}=M\mathbb{R}_d^+ \cap \mathring{\Delta}_d$. 

\begin{lem}\label{region} 
  Let $T$ be some IET. An IET $S$ with permutation $\pi(S)=\pi(T)$ and
  whose length vector $L(S)$ is contained in $M(T,1)\Delta$ has the
  same first step of Rauzy induction as $T$.
\end{lem}

We define the $n^{\text{th}}$ matrix of Rauzy induction
by $$M(T,n)=M(T,n-1)M(R^{n-1}(T),1).$$  
 It follows from Lemma \ref{region} that for an IET with length vector
 in $M(T,n)\Delta$ and permutation $\pi$ the first $n$ steps of
 Rauzy induction agree with $T$. 

\medskip
The set of all normalized $n$-IETs with a given permutation $\pi$ can be
naturally identified with a simplex which we denote by $\Delta_\pi$.
Iterated Rauzy induction defines a partition of $\Delta_\pi$ into
smaller simplices in the following way. The subset of $\Delta_\pi$
corresponding to IETs on which Rauzy induction is undefined is a
codimension-$1$ simplex embedded in $\Delta_\pi$. In each of the
complementary full-measure simplices, the set where Rauzy induction is
defined once but not twice again is a union of two codimension-$1$
simplices.
We denote by $\mathcal{P}_k$ the full-measure partition of
$\Delta_\pi$ into the (open) simplices on which Rauzy induction is
defined for the first $k$ steps. A simplex in $\mathcal{P}_k$ consists
of all the IETs which follow the same first $k$ steps in the Rauzy
diagram under Rauzy induction.

The projective linear map defined by the matrix $M(T,k)$ maps the
standard simplex $\Delta_{\pi'}$ (where $\pi'$ is 
the permutation corresponding to the IET $R^k(T)$) to the 
simplex in $\mathcal{P}_k$ which contains $T$.

We will also need a criterion that ensures unique ergodicity of
IETs. A \emph{Rauzy path} is a finite or infinite path in a Rauzy diagram
$\mathcal{R}$. Associated to a Rauzy path is the product of matrices
describing the change on lengths of the intervals which we will call
the \emph{Rauzy matrix} of the path. The following is well-known.
\begin{thm}\label{thm:ue-criterion}(Veech \cite[page 225]{iet v})
  Suppose that $T$ is an IET where $R^n(T)$ is defined for all
  $n\geq 1$, and such that $\bigcap_{n\geq 1}M(T,n)\Delta =
  \{T\}$. Then $T$ is uniquely ergodic.
\end{thm}
To check that the prerequisite of Theorem~\ref{thm:ue-criterion} is
satisfied, it suffices to check that the angle between the columns of
$M(T,n)$ converges to $0$, since $M(T,n)\Delta$ consists of convex
combinations of the columns of $M(T,n)$.
\subsection{Shadows of Rauzy paths} \label{sec:shadows}
We need to make the following consideration for our proof to hold for all uniquely ergodic 4-IETs, as opposed to just
 uniquely ergodic 4-IETs that have all powers of Rauzy induction defined. 
 
 Let $T$ be a minimal $d$-IET with the length of an interval equal to $0$. Let $\hat{T}$ be the minimal $d-1$-IET given by ``forgetting the interval of $T$ with length 0." For example if $L(T)=(\alpha,0,1-\alpha)$ and $\pi(T)=(321)$ then $L(\hat{T})=(\alpha,1-\alpha)$ and $\pi(\hat{T})=(21)$.
   Assume that $R^n(\hat{T})$ is defined for all $n$. 
  Let $T_\epsilon$ be the $d$-IET with $\pi(T_\epsilon)=\pi(T)$ and 
  \[L(T_\epsilon)= \begin{cases}L(T) & \text{ if }L(T)\neq 0\\
  \epsilon & \text{ else} \end{cases}.\]
  For all $\ell>0$ there exists $n$ (possibly 2 different $n$) and $\sigma_n \in \{1,...,d\}$ so that  
    \begin{itemize}
 \item $\underset{\epsilon \to 0}{\lim} \, L(R^nT_\epsilon)_i=L(R^\ell \hat{T})_{i-c_{i,\sigma_n}}$ for all $i$
\item For any $i,j\neq \sigma_n$ we have $\pi(R^n(T_\epsilon))(i)>\pi(R^n(T_\epsilon))(j)$ iff $\pi(R^\ell(\hat{T}))(i-c_{i,\sigma_n})>\pi(R^\ell(\hat{T}))(j-c_{j,\sigma_n})$ for all small enough $\epsilon$
  \end{itemize}
  where $c_{k,\ell}=\begin{cases}
  0 &\text{ if } i<\ell\\
  1 & \text{ if } i>\ell\\
  \text{ undefined} & \text{ if }i=\ell
  \end{cases}.$ 
  
 Given $n$, for all $\epsilon$ small enough $\epsilon$ we have $M(T_\epsilon, n)$ and $M(\hat{T},\ell)$ are related in the following way: after deleting the $\sigma_n^{\text{th}}$ column of $M(T_\epsilon,n)$ one is left with a $d \times (d-1)$ matrix with a row of all zeros, so that when it is deleted one has $M(T,\ell).$

 We prove the result by induction on $\ell$. We assume that this is true for $\ell$ with corresponding number $n$. We say $j$ is in \emph{critical position} if $j\in \{d, \pi^{-1}d\}$. The proof breaks into two parts, when $\sigma_n$ is not in critical position in which case the corresponding number for $\ell+1$ is $n+1$ and when $\sigma_n$ is in critical position in which case the corresponding number for $\ell+1$ is $n+2$.

\smallskip
\textbf{Case 1:}  $\sigma_n$ is not in critical position. In this case it is straightforward for $\ell+1$ and the corresponding number is $n+1$ with 
 \[\sigma_{n+1}=\begin{cases}
 \sigma_n & \text{ if }R (R^nT_\epsilon) \text{ is `b' or } R(R^nT_\epsilon) \text{ is `a' and }\sigma_n<\pi^{-1}(d)\\
 \sigma_n+1 & \text{ else}
 \end{cases}.\]

\smallskip
  \textbf{Case 2:} $\sigma_n$ is in a critical position. Let $j=d$ if $\sigma_n=\pi^{-1}(d)$ and $\pi^{-1}(d)$ if $\sigma_n=d.$ By Conclusion 1 (which we are inductively assuming), for all small enough $\epsilon$ we have that $L(T_\epsilon)_j>L(T_\epsilon)_{\sigma_n}=\epsilon$. Set 
  \[\sigma_{n+1}=\begin{cases}
 \pi^{-1}(d)+1 & \text{ if }\sigma_n=d\\
 \sigma_n & \text{ else}
 \end{cases}.\] Conclusion 1 follows inductively. Indeed, except for one entry $L(R^{n+1}T_\epsilon)$ is a permutation of $L(R^nT_\epsilon)$ and the exceptional entry is $L(R^nT_\epsilon)_j-\epsilon$. Conclusion 2 is straightforward to check, as is the claim on the matrices (the deleted column should be $\sigma_{n+1}$). Moreover $\sigma_{n+1}$ is not in the critical position so we may follow case 1 for $\ell+1$ which will correspond to $n+2$. To see that $\sigma_{n+1}$ is not in critical position, if $\sigma_n=d$ then $\pi^{-1}(d)\neq d-1$ be the minimality of $R^\ell\hat{T}$ (which follows from our assumptions on $\hat{T}$). The case of $\sigma_n=\pi^{-1}(d)$ is similar.
 
 This discussion has a bearing on minimal IETs with nonzero interval lengths that will have a failure of Rauzy induction. Indeed, assume $T$ is a minimal $d$-IET
  that has  $R^k$ defined on it, but $R^{k+1}$ is not defined on it. This means that the two intervals of $R^kT$ in critical position have the same length. One can formally continue Rauzy induction (by either $R_a$ or $R_b$) one step and have an IET, $S$ with the length of one entry 0. One also gets a matrix $M=M(T,k)M'$ where $M'$ is the matrix given by the formal step or Rauzy induction. We may now forget the entry of $S$ that has zero length and obtain a $(d-1)$-IET, $\hat{S}$. We may relate its path under Rauzy induction to $S_{\epsilon}$ as in the previous discussion. To recover an approximation to $T$ we just multiply the matrices and lengths we obtain for $S$ by $M$. This procedure can be successively iterated if there is a subsequent failure of Rauzy induction for $\hat{S}$.

\section{Matrix Terminology and Combining Matrices}
\label{sec:matrices}
In this section we will collect some notation and terminology on 
$(4\times 4)$--matrices with nonnegative integral entries. 

If $A$ is such a matrix, then we will denote by $C_i(A)$ the $i$-th
column of $A$. We will often call $i$ the \emph{index} of the column
in such a context.  If $P$ is some property a column of a matrix (or a
vector in general) may have, we will say that the \emph{index $i$ or
  column $i$ has $P$} if $C_i(A)$ has $P$, when the matrix $A$ is clear
from the context.

\smallskip
We use the $1$-norm for vectors throughout, so 
$|C_i(A)|$ will denote the sum of the entries in the $i$-th column.
We will often need to consider the column of $A$ with the largest sum
of entries. We denote this by $C_{max}(A)$. Should there be several
columns which all have the largest entry sum, we adopt the convention
that $C_{max}(A)$ is the column with smallest index realising this maximum.

\medskip
Recall that in a matrix product $AB$, the column $C_i(AB)$ is a sum
of the columns of $A$, with coefficients from $C_i(B)$:
\[ C_i(AB) = \sum_{j=1}^4 (C_i(B))_j C_j(A) \]
Thus, we say that \emph{$B$ adds column $i$ to column $j$ (or: column
  $i$ is added to column $j$)} if the $i$-th entry of $C_j(B)$ is
positive (usually, it will be the case that the $j$-th entry of
$C_j(B)$ is also positive, to make the terminology completely
justified, but we do not insist on this). Note that it is possible for
a column to add to itself.  Equivalently, we may say that \emph{column
  $j$ has column $i$ added to it}.  If there is any $j$ so that column
$i$ is added to column $j$, then we simply say that \emph{column $i$
  is added to another column}. Similarly, we simply say that
\emph{column $i$ has a column added to it} if there is a corresponding
$j$.

This terminology will be used in particular when considering powers of
some matrix $B$. 

\medskip Next, we introduce the central new notion of this
section. Combining matrices will be generalisations of positive
matrices, in the sense shown in Lemma~\ref{lem:powers-of-combining}
below.

\begin{defin} 
  A matrix $M$ is called \emph{combining} if there are two groups of columns,
  \emph{active} and \emph{passive} ones. We require that that 
  \begin{enumerate}
  \item There are at least $2$ active columns.
  \item At most one column of $M$ is neither active nor passive. This
    column is called \emph{idle} if it exists.
  \item The only columns that are added to other columns are the
    active ones.
  \item Every active column is added to all other active columns and
    each passive column has at least one active column added to it.
  \end{enumerate}
  A finite Rauzy path $P$ is said to be \emph{combining} if
  the corresponding matrix is combining.
\end{defin}
\begin{example}
  \begin{enumerate}
  \item Every positive matrix is combining so that every column is active.
  \item The matrix 
    $$\begin{pmatrix}
      1 & 1 & 1 & 1 \\
      0 & 1 & 0 & 0 \\
      0 & 0 & 1 & 0 \\
      1 & 2 & 2 & 2
    \end{pmatrix}$$
    is combining. Namely, columns 1 and 4 are active, while 2 and 3 are passive.
  \item The matrix 
    $$\begin{pmatrix}
      1 & 1 & 1 & 1 \\
      0 & 1 & 0 & 0 \\
      0 & 1 & 1 & 0 \\
      1 & 2 & 2 & 2
    \end{pmatrix}$$
    is \textbf{not} combining. Column 2 has Columns 1,3 and 4 added to
    it. Thus, all three of those would need to be active. However, column 4 
    does not have column 3 added to it, violating the definition of active.
  \item The matrix $$\begin{pmatrix}
      1 & 0 & 0 & 0 \\
      0 & 1 & 0 & 0 \\
      0 & 0 & 2 & 1 \\
      1 & 0 & 1 & 1
    \end{pmatrix}$$
    is combining, where columns 3 and 4 are active, 1 is passive and
    column 2 is idle.
  \end{enumerate}
\end{example}

\begin{thm*}(Perron-Frobenius) If $M$ is an $n \times n$
  matrix with all entries positive then there exists a unique largest
  (in absolute value) eigenvalue, called the Perron-Frobenius
  eigenvector. The corresponding eigenvector can be chosen to be
  positive (i.e. has all entries positive) and is called the
  Perron-Frobenius eigenvector. It is the only eigenvector with that
  property.
\end{thm*}
One can observe that matrices as in the previous theorem act as a
contraction in the so called Hilbert projective metric and we obtain
the following result (see e.g. \cite[page 240]{iet v}).
\begin{prop}\label{prop:hilbert}
  If $M$ is an $n \times n$ matrix with all entries positive then any
  non-negative (non-zero) vector $v$ has that $M^nv$ converges
  exponentially quickly to the direction of the Perron-Frobenius
  eigenvector.
\end{prop}

We conclude with a lemma generalizing the Perron-Frobenius theorem to the 
case of combining matrices.
 It will be used frequently in the sequel.

In its formulation, we will call the $j$--th entry $(C_i(A))_j$ of a column
\emph{active} (or passive, or idle), if the corresponding index $j$ has this
property (i.e. the column $C_j(A)$ is active, passive, or idle).
\begin{lem}\label{lem:powers-of-combining}
  Let $A$ be a combining matrix. Then there are numbers $E,\gamma'>1$ with the
  following properties: let $n>3$ be any integer, and let $i,j$ be
  indices corresponding to active or passive columns.
  \begin{enumerate}[i)]
  \item $|C_i(A^n)| / |C_j(A^n)| \leq E$.
  \item Any two of the active entries in columns
    $i$ or $j$ of $A^n$ differ by a factor of at most $E$.
  \item We have
    \[ \sin\angle(C_i(A), C_j(A)) \leq \frac{E}{\gamma'^n} \]
  \end{enumerate}
\end{lem}

The proof requires the following standard lemma.
\begin{lem}[Law of Sines]\label{lem:add vectors} Let $v, w \in \mathbb{R}^n,$ $$| \sin\angle(v+w,w)|
  |= \frac{\|v\|}{\|v+w\|} |\sin \angle(v,w)|.$$\end{lem}

\begin{proof}[Proof of Lemma~\ref{lem:powers-of-combining}]
  Suppose that $\tau\leq 4$ columns are active, and let
  $Y\subset\RR^4$ be the $\tau$-dimensional subspace which is
  preserved by $A$ and so that the matrix $\hat{A}$ describing the
  action of $A$ on $Y$ is positive. Thus by the Perron-Frobenius
  theorem there is a unique positive eigenvector $\hat{w}\in Y$ of
  $\hat{A}$, which has positive eigenvalue $\mu$ that is the largest
  eigenvalue in absolute value. If $\tau<4$, consider the action of
  $\hat{A}$ on the invariant subspace of dimension $\tau-1$ which does
  not contain $\hat{w}$, $\tilde{A}$.  There exists $\gamma<\mu$ and
  $C$ so that $\|\tilde{A}^k\|_{op}\leq C\gamma^k$.

\begin{enumerate}[i)]

\item 
  By the Proposition~\ref{prop:hilbert} it follows that the projection
  of the active columns to the $\tau$ dimensional subspace in the
  previous paragraph does not lie in the $\hat{A}$ invariant $\tau-1$
  dimensional subspace complemented to the Perron-Frobenius
  eigenvector.  So the size of these columns grow proportionally to the
  Perron-Frobenius eigenvalue of $\hat{A}$. The passive columns get
  some multiples of active columns added to them. These active columns
  are growing exponentially (according to the Perron-Frobenius
  eigenvalue) and so the column is proportional to the last summand
  once $n$ is large enough. For small $n$ proportionality also holds
  simply by finiteness. If we consider a passive column $v$ of $A^n$
  it has the form $\sum_{j=1}^{n-1}w_j$ where the $w_j$ are active
  columns of $A^j$. Since the active columns are growing
  exponentially, $w_{n-1}$ is proportional to the largest active
  column of
  $A^n$.

\item This is true for powers of the $\tau$-by-$\tau$ matrix $\hat{A}$.
Indeed by Proposition~\ref{prop:hilbert} the ratio of
the entries converges to the ratio of the entries of the
Perron-Frobenius eigenvector. Arguing as in $(i)$ the passive 
column(s) are sums of vectors with this property, so they inherit it. 

\item 

Arguing as before, by Proposition \ref{prop:hilbert} the active columns of $A$ converge
exponentially fast to $\hat{w}$ under taking powers, and the angle
between the active columns decreases as required.

Consider next a passive column of $A^n$. This column has the form
$\sum_{i=0}^nv_i$, where $v_0$ is the initial passive column of $A$,
and each $v_i$ is a linear combination of active columns of $A^i$. In
particular, each $v_i$ has norm between $D_1\mu^i$ and $D_2\mu^i$
where $\mu$ is the Perron-Frobenius eigenvalue and $D_1,D_2$ depend on
the matrix $A$.

Next, note that
\[ \sin\angle\left(\sum_{i=0}^nv_i, \hat{w}\right)
\leq \sin\angle\left(\sum_{i=0}^n v_i, \sum_{i=n/2}^n v_i\right) +
\sin\angle\left(\sum_{i=n/2}^n v_i,\hat{w}\right) \]
since (absolute values of) angles satisfy the triangle inequality
and $\sin(\theta+\phi)\leq \sin(\theta)+\sin(\phi)$ for all
$0\leq \theta,\phi$ with $\theta,\phi<\frac \pi 2$.

Since
$\frac{\sum_{i=1}^{\frac n 2} D_2\mu^i}{\sum_{i=\frac n 2}^nD_1\mu^i}$
decays exponentially in $n$, Lemma~\ref{lem:add vectors} implies that
the first summand decays exponentially as well.

On the other hand, as $\sum_{i=n/2}^n v_i$ is a positive linear
combination of the active columns of $A^{n/2}$, the second summand
decays exponentially as well, by the previous comments.
\end{enumerate}
\end{proof}

\section{Building blocks}
\label{sec:abstract}
This section sets up most of the novel terminology used to construct the paths
necessary to prove the following
\begin{thm}\label{thm:main-for-4}
  The set of uniquely ergodic $4$-IETs with permutation in the Rauzy
  class of $(4321)$ is path-connected.
\end{thm}
The proof will involve an explicit construction of paths by prescribing
Rauzy inductions. The main tool in this construction is given by the following
definition.

\begin{defin} 
  A \emph{building block} $b$ is a triple of IETs $b=(T_1,F,T_2)$ such that  
  \begin{enumerate}
  \item $T_1, F, T_2$ are uniquely ergodic with the same underlying
    permutation $\pi$.
  \item $F$ lies on the plane given by the failure of the first step of
    Rauzy induction.
  \end{enumerate}
  We call $T_1,T_2$ the \emph{endpoints} of the building block, and
  $F$ the \emph{midpoint}.
\end{defin}
The next definition lies at the core of our argument. Intuitively we would
like to say that a building block $b'=(S_1,G,S_2)$ is left compatible
with $b=(T_1,F,T_2)$ if the IETs $T_1$ and $F$ share a number of
common Rauzy steps, and the outcome is the pair $S_1,S_2$ of IETs. We
formally need to phrase this slightly differently to avoid the problem that
Rauzy induction is not well-defined for $F$.
\begin{defin}
  A building block $b'=(S_1,G,S_2)$ is \emph{left compatible} with
  a building block $b=(T_1,F,T_2)$ if the following holds:
  \begin{enumerate}
  \item $S_1=R^{k_1}(T_1)$.
  \item Every point on the straight line between $T_1$ and $F$ has
    Rauzy induction defined for $k_1$ steps, and it agrees with the
    one of $S_1$.
  \item $L(F) = M(T_1, k_1)L(S_2)$
  \end{enumerate}
  We then write $b\stackrel{L}{\to} b'$. 
  
  The \emph{Rauzy matrix associated to the compatability
    $b\stackrel{L}{\to} b'$} is then defined to be the matrix
  $M(T_1^{\text{right}},k_1)$ where $T_1^{\text{right}}$ denotes the
  limit from the right..

  \smallskip We define right compatible and $\stackrel{R}{\to}$
  similarly. We take \emph{compatible} and $\to$ to mean left compatible or
  right compatible.
\end{defin}

\begin{figure}
  \centering
  \includegraphics[width=0.75\textwidth]{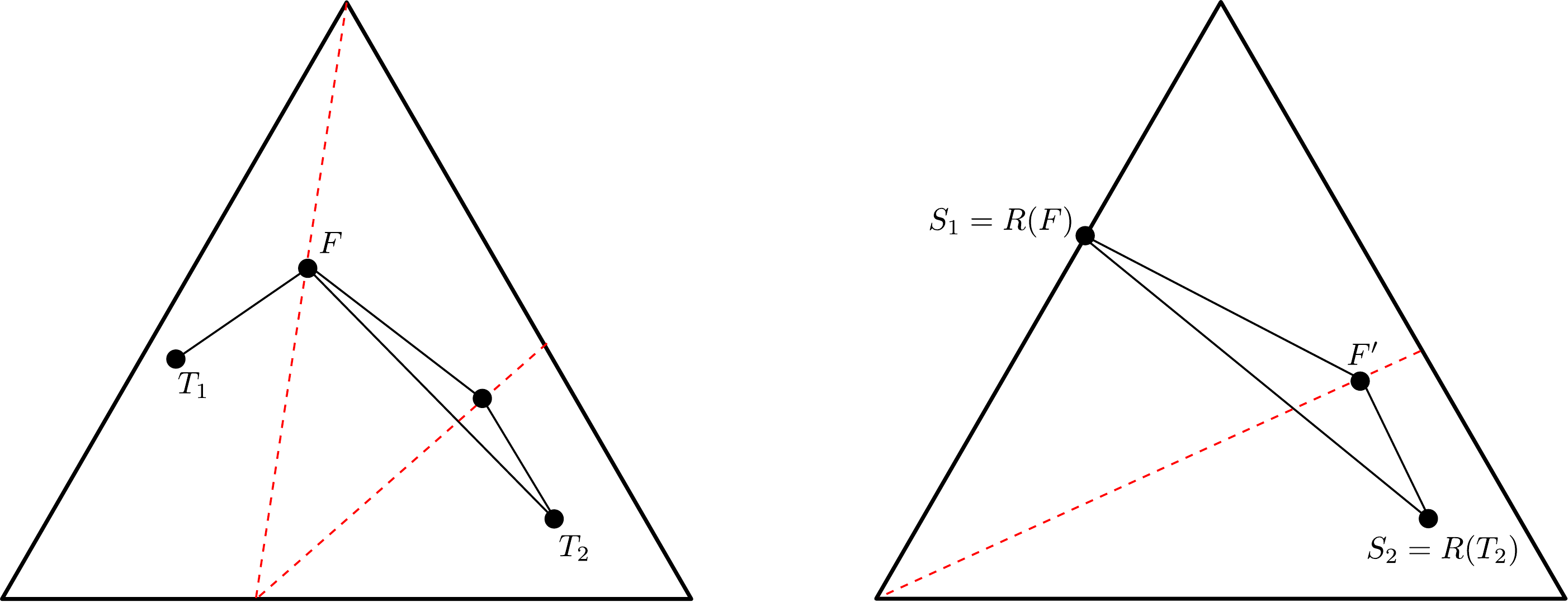}
  \caption{Building blocks. For clarity, the dimension is reduced from
  3 to 2. The dotted lines represent fail planes of Rauzy induction.}
  \label{fig:building-blocks}
\end{figure}
\begin{defin}
  A \emph{building block sequence} is a (finite or infinite) sequence
  $b_i = (T_1^i,F^i,T_2^i)$ of building blocks so that $b_{i+1}$ is
  compatible with $b_i$ for all $i$.  We often write such a sequence as
  $b_1\to b_2\to \dots$.

  For a finite building block sequence $b_1\to\dots\to b_n$ we define the
  \emph{associated Rauzy matrix} $M(b_1\to\dots\to b_n)$ as the product
  $M=M_1\dots M_{n-1}$ where $M_i$ is the Rauzy matrix associated to
  the compatibility $b_i\to b_{i+1}$.
\end{defin}
The following definition of depth of a building block sequence is
crucial for our construction. 

Intuitively, we want to define the depth
of a building block sequence as the number of times it swaps between
consecutively always choosing right and always choosing left.

To state it formally, note that a building block sequence
$b_1\to\dots\to b_n$ defines a ``direction sequence''
$d_1,\ldots, d_{n-1}$, where $d_i \in \{L,R\}$ and $d_i = L$ if and
only if $b_{i+1}$ is left compatible with $b_i$. We say that $d_i$
\emph{swaps direction at most $s$ times} if there are indices
$1=t_0<t_1<\dots<t_{s+1}=n$ so that $d_i$ is constant for
$t_j\leq i < t_{j+1}$. 

\begin{defin}
Let $b_1\to\dots\to b_n$ be a building block sequence
The \emph{depth} of the building block sequence is the smallest $s$ so that
the associated direction sequence $d_i$ swaps direction at most $s$ times.
\end{defin}
As an important example, a building block sequence $b_1\to\dots\to b_n$
has depth $0$ if either for all $i$, $b_i$ is left compatible to
$b_{i-1}$ or for all $i$, $b_i$ is right compatible to
$b_{i-1}$.

Intuitively, the sequence $b_1\to\dots\to b_n$ has depth $k$ if the compatibilities $b_i\to b_{i+1}$ swap $k$ times between choosing L or R.

\begin{defin}
  \begin{enumerate}[i)]
  \item A \emph{building block loop} is a building block sequence
    $b_1\to \dots\to b_n$ such that the last element $b_n$ is
    compatible with the first element $b_1$.

  \item Such a loop is \emph{minimal}, if $b_i \neq b_j$ for all
    $1\leq i < j \leq n$.

  \item The \emph{depth} of a building block loop $b_1\to \dots\to b_n$ is
    defined to be the depth of the building block sequence
    $b_1\to \dots\to b_n\to b_1$.
  \item If $b_1\to \dots\to b_n$ is a depth $0$ building block loop, then one
    of the endpoints of $b_1$ is contained in
    $M(b_1\to\dots\to b_n\to b_1)\Delta$. We call that IET the
    \emph{endpoint of the building block loop}.
  \end{enumerate}
\end{defin}

\medskip
We will frequently use product notation for building block
sequences. That is, let $P$ is a building block sequence $b_1\to \dots\to
b_n$ and $Q$ is a building block sequence $b_{n+1}\to\dots\to b_m$.
We say that $P$ is compatible with $Q$ if $b_n$ is compatible with $b_{n+1}$.
In that case we denote by $P\to Q$ the sequence 
$$b_1\to \dots\to b_n \to b_{n+1}\to\dots\to b_m$$
obtained by concatenating $P$ and $Q$.  Similarly, if $P$ is a
building block loop $b_1, \ldots, b_n$, the $P^n$ is the sequence
$P\to \dots \to P$ obtained by concatenating $P$ with itself $n$
times.

\begin{defin}
  If $P$ is a building block loop $b_1\to\dots\to b_n$, then we say
  that a building block sequence $c_1\to c_2\to c_3$ compatible with $P$
  \emph{leaves the loop $P$} if the sequence $c_1\to c_2\to c_3$ does not
  contain $P$ as an initial segment, and is not contained in $P$ as an
  initial segment. 
\end{defin}

In Section~\ref{sec:concrete} we will show the following Proposition
by an explicit construction.
\begin{prop}\label{prop:existence-of-building-blocks}
  There is a finite set $\mathcal{B}$ of building blocks with the
  following properties 
  \begin{description}
  \item[Transitivity] For every permutation $\pi$ in the Rauzy class
    of $(4321)$ there is a building block in the set $\mathcal{B}$
    whose endpoints lie on different sides of the fail plane.
  \item[Completeness] Every building block in $\mathcal{B}$ has a left and
    right compatible building block in the set $\mathcal{B}$.
  \item[Combining Loops] If $b_1\to\dots\to b_k$ is a minimal depth
    $0$ loop formed from building blocks in the set $\mathcal{B}$, then
    the matrix $M(b_1\to\dots\to b_k)$ is combining.
  \item[Isolated Idle] Suppose $P$ is a building block loop of depth
    $0$ formed by building blocks in the set $\mathcal{B}$, so that
    $M(P)$ has an idle column $i$. Then, for any building block sequence
    $b_1\to b_2\to b_3$ which leaves $P$ formed by building blocks in the set
    $\mathcal{B}$, each column of
    $M(P\to b_1\to b_2\to b_3)$ has at least two nonzero entries.
  \item[Almost Positivity] There is a number $c \geq 0$ so that every
    building block sequence formed by building blocks in the set
    $\mathcal{B}$ of depth at least $c$ has an initial segment
    $b_1\to \dots\to b_k$ with $k\leq c$ so that
    $M(b_1\to\dots\to b_k)$ is \emph{Almost Positive}: i.e. it has
    $\tau>1$ rows with all entries positive and the other rows are
    rows of the identity matrix. By this we mean that they contain
    exactly one entry $1$, and all other entries are $0$.
  \end{description}
\end{prop}
In Sections~\ref{sec:finite-depth}
to~\ref{sec:connecting-building-blocks} we will show how to use
Proposition~\ref{prop:existence-of-building-blocks} to show
Theorem~\ref{thm:main-for-4}. Hence, we fix once and for all a finite
set $\mathcal{B}$ as given by
Proposition~\ref{prop:existence-of-building-blocks}. Any mention of
building blocks will refer to elements of this set. We will also refer
to the properties guaranteed by
Proposition~\ref{prop:existence-of-building-blocks} by their name,
without explicit mention of
Proposition~\ref{prop:existence-of-building-blocks}.

Following the sketch outlined in the
introduction there is a general dichotomy of points on the paths: finite depth
points eventually follow the Rauzy induction of a fixed uniquely
ergodic IET (which is the left or right endpoint of a building block). Continuity at these points will be shown using the tools
in Section~\ref{sec:finite-depth}. Infinite depth points on the
other hand are those that do not eventually lie on a fail plane, and
whose Rauzy induction does not eventually follow one of the
endpoints. For these points, convergence, 
unique ergodicity and continuity all need to be checked, using
techniques developed in Section~\ref{sec:infinite-depth}. Finally,
in Section~\ref{sec:connecting-building-blocks} we collect the
pieces and prove Theorem~\ref{thm:main-for-4}.

\section{Finite Depth}\label{sec:finite-depth}

\begin{lem}\label{lem:convergence-most-columns}
  Let $P = b_1 \to\dots\to b_k$ be a minimal depth $0$ loop. Denote by $b$ the
  index of the column which is idle if it exists. We consider the
  building block sequence $P^n$. Then there exists a vector
  $V$ such that 
  $$\frac{C_i(M(P^n))}{|C_i(M(P^n))|} \to V$$
  for all $i\neq b$. 

  Furthermore, the size of the columns $i\neq b$ grow exponentially,
  and the size of column $b$ is constant.
\end{lem}
\begin{proof}
  This is immediate from Lemma~\ref{lem:powers-of-combining} and
  \textbf{Combining Loops}.
\end{proof}
\begin{lem}\label{lem:finite-depth-convergence}
  Let $P$ be a minimal depth $0$ loop, and let $P_0$ be a building
  block sequence compatible with $P$. 
  We consider building block sequences of the form
  $P_k = P_0\to P^{n_k}\to Q_k$ where $Q_k$ is a building block sequence of
  length $3$ which leaves the loop $P$ and
  $\underset{k \to \infty}{\lim}\, n_k=\infty$.

  Let $T_\infty$ be the IET corresponding to the endpoint of $P$
  defined by the loop. Let $\{T_i\}_{i\in\NN}$ be IETs so that
  $T_k \in M(P_k)\Delta$ (i.e. the Rauzy induction of $T_k$ agrees with the
  one defined by $P_k$)

  Then the $T_k$ converge (as a sequence of points in the simplex) to
  $T_{\infty}$.
\end{lem}
\begin{proof}
  To begin, note that we may assume that $P_0$ is empty. Namely, the
  Rauzy induction corresponding by $P_0$ defines a continuous map
  $M(P_0)$ of the simplex which does not change the result.

  We fix $m_k>0$ so that the Rauzy expansion of the building block
  sequence $P^{n_k-1}$ has $m_k$ Rauzy steps. Further, let $m'_k$ be the
  number of Rauzy steps in $P\to Q_k$. Thus, the initial $N_k = m_k+m'_k$ Rauzy
  steps of the interval exchange $T_k$ agree with the Rauzy expansion defined
  by $P_k$.

  Next note that the (unnormalized) length vector of $T_k$ satisfies
  $$\hat{L}(T_k) = M\left(P^{n_k}Q_k\right) \hat{L}\left(R^{N_k}(T_k)\right)$$
  The length vector for $T_k$ is $L(T_k) = \hat{L}(T_k)/|\hat{L}(T_k)|_1$.

  Write $L\left(R^{m_k}(T_k)\right) = \sum_{i=1}^4 a_i e_i$. We have
  $\sum_{i=1}^4a_i = 1$. We let $v_i$ denote the $i$-th column of $M(P^{n_k-1})$. 
  Note that in both of these shortcuts we are
  suppressing the dependence of the $a_i$ and the $v_i$ on $k$. We
  have
  $$L(T_k) = \frac{\sum_{i=1}^4 a_i v_i}{\sum_{i=1}^4 a_i |v_i|_1}$$

  Next, we claim that there is a number $\epsilon_0 > 0$ (depending
  only on the set of building blocks and not $k$) such that at least
  two of the $a_i$ are at least $\epsilon_0$.  Namely, we have
  \[ (a_1, \ldots, a_4) = M(P\to Q_k)v/|M(P\to Q_k)v|_1 \]
  for some vector $v$ with non-negative entries and $|v|_1=1$. Hence,
  at least one entry $v_i$ of $v$ has size at least $1/4$.  By
  Isolated Idle, each column of $M(P\to Q_k)$ has at least two
  positive entries, and hence $M(P\to Q_k)e_i$ has at least two
  nonzero entries for each $i$. Thus, at least two entries of the
  (unnormalized) vector $M(P\to Q_k)v$ have size at least
  $1/4$. On the other hand, note that $M(P\to Q_k)$ is a
  matrix whose entries can be bounded by some constant which does not
  depend on $k$ (since the set of building blocks is finite). Thus,
  $|M(P\to Q_k)v|_1$ can likewise be bounded by some constant
  independent of $k$. This implies the existence of $\epsilon_0$ as
  claimed.

  Now suppose that we are given some
  $\epsilon > 0$. We choose $K$ so that for all $n>K$ the following hold
  (where $V$ is the vector given by Lemma~\ref{lem:convergence-most-columns})
  \begin{enumerate}[i)]
  \item For all but at most one $i$, we have
    $$|v_i - |v_i|_1V|_1 < \epsilon |v_i|_1$$
    We call the $i$ where this fails the \emph{bad $i$} and $v_i$ the \emph{bad column}. We call other indices $j\neq i$ and the corresponding columns $v_j$ \emph{good}. 
    This property
    can be ensured by Lemma~\ref{lem:convergence-most-columns}.
  \item 
    If there is a bad column $v_i$, then
    $$\frac{a_iv_i}{\sum_{j \mathrm{good}} a_j |v_j|} < \epsilon.$$
    This is possible since at least one coefficient $a_j$ of a good
    column $v_j$ is at least $\epsilon_0$, and the sizes of the good
    columns grow exponentially, while the size of the bad column
    is uniformly bounded by Lemma~\ref{lem:convergence-most-columns}.
  \end{enumerate}
  Now, let $n>K$ be given.

  \smallskip
  We first consider the case in which there is no bad $i$ for $M(P^N)$. 
  In that case, we compute  
  $$|L(T_k) - V|_1 = \left|\frac{\sum_{i=1}^4 a_i v_i - \sum_{i=1}^4 a_i
      |v_i|_1 V}{\sum_{i=1}^4 a_i |v_i|_1}\right|_1 \leq
  \frac{\sum_{i=1}^4 a_i \left|v_i - |v_i|_1 V\right|_1}{\sum_{i=1}^4
    a_i |v_i|_1}.$$ 
  Thus by i) 
  $$|L(T_k) - V|_1 < \frac{\sum_{i=1}^4 a_i|v_i|_1\epsilon}{\sum_{i=1}^4
    a_i |v_i|_1} = \epsilon.$$ 
  Now suppose that there is a bad $i$, which without loss of
  generality we may assume to be $1$.

  In this case, we compute
  $$L(T_k) = \frac{\sum_{i=1}^4 a_i v_i}{\sum_{i=1}^4 a_i |v_i|_1}
  = \frac{a_1 v_1}{\sum_{i=1}^4 a_i |v_i|_1} 
  + \frac{\sum_{i=2}^4 a_i v_i}{\sum_{i=1}^4 a_i |v_i|_1}.$$
  By property ii), the first summand has size at most $\epsilon$. Now
  note that for any numbers $K<D$ and $c$ one has
  $$\frac{K}{D} - \frac{K}{c+D} = \frac{cK}{D(c+D)} < c \frac{K}{D} \frac{1}{D}
  < c \frac{1}{D}.$$
  We apply this estimate for $K$ being each of the four entries of the
  vector $\sum_{i=2}^4 a_i v_i$ with
  $D = \sum_{i \mathrm{good}} a_i |v_i|_1$ and $c=a_1|v_1|_1$ to
  conclude that the second summand in the expression for $L(T_k)$ is
  within $4\epsilon$ of
  $$\frac{\sum_{i=2}^4 a_i v_i}{\sum_{i=2}^4 a_i |v_i|_1}.$$
  Arguing as in the first case, this is within $\epsilon$ of $V$. In
  conclusion, we thus have $|L(T_k) - V|_1 < 6\epsilon$.

  Thus, the normalized length vectors $L(T_k)$ converge to $V$.
\end{proof}

\begin{cor}\label{cor:non idoc} Let $P$ be an infinite Rauzy path so
  that at least 3 columns increase in size an unbounded 
  amount\footnote{Note that all matrices of Rauzy induction have only non-negative entries and thus such an increase is automatically monotone.} and 
  all columns which do increase an unbounded amount converge projectively to some fixed vector
  $V$. Let $Q_1,...,Q_k$ be Rauzy paths with positive associated Rauzy
  matrices. 

  Then for every $\epsilon>0$ there is a $N>0$ so that if $n>N$, and 
  $T$ is any IET whose initial $n$-steps of Rauzy
  induction agree with $P$ and then are followed by some $Q_i$, then
  $L(T)$ is within $\epsilon$ of $V$.
\end{cor}
\begin{proof}
  By assumption of the Corollary, for any $\epsilon$ there is an $N$
  so that for all $n>N$ conditions i) and ii) stated in the proof of
  Lemma~\ref{lem:finite-depth-convergence} hold for the column $v_i$
  of the Rauzy matrix $M(P_n)$ describing the first $n$ steps of $P$
  (where $V$ is now the vector given by the assumption).

  Furthermore, as $Q_i$ is assumed to be positive, $L(R^N(P_n))$ has
  at least two entry of size at least $\epsilon_0$ (arguing as in the
  previous proof, with positivity of $Q_i$ replacing the part of the argument
  involving Isolated Idle).

  From here, one can finish the proof exactly as in the proof of 
  Lemma~\ref{lem:finite-depth-convergence} above.
\end{proof}

\section{Infinite Depth}\label{sec:infinite-depth}
Next, we analyse infinite depth points. Here, the situation is much
more involved. Convergence, continuity and unique ergodicity will all
follow from the same mechanism, which we now describe. 

Throughout, we consider an infinite building block sequence $(b_i)$
of infinite depth. This in particular means that the sequence $b_i$
does not end in an infinite power of a depth $0$ loop.

We rewrite the sequence $b_1\to b_2 \to \dots$ as a sequence of paths
$B_1\to B_2 \to \dots$, where each $B_i$ is either a power of a
minimal depth $0$ loop, or a single building block, and the
re-grouping is maximal with that property (that is: if $B_i$ is a
power of a minimal depth $0$ loop then the sequence following $B_i$
leaves the loop). By our assumption, each $B_i$ is a finite building
block sequence, and the sequence of $B_i'$ is infinite.

Denote throughout this section by $A_i=M(B_i)$ the matrix
corresponding to $B_i$.

By Proposition~\ref{prop:existence-of-building-blocks} these matrices
satisfy the following:
\begin{cor}\label{cor:matrix-properties}
There are numbers $c,N>0$ such that for all $i$ the following holds:
\begin{enumerate}
\item[(P)] The product $A_i \cdot...\cdot A_{i+c}$ is positive  or
\item[(R)] There exists $c'\leq c$ so that $A_i\cdot....\cdot
  A_{i+c'}$ has $\tau>1$ rows with all entries positive and the
  other $(4-\tau)$ rows are rows of the identity matrix.
\end{enumerate}
Furthermore, each matrix $A_i$ satisfies one of the following.
\begin{enumerate}
\item[(S)] $||A_i||<N$ or
\item[(C)] $A_i$ is the power of a combining matrix $B$ with $\|B\|<N$.
\end{enumerate}  
\end{cor}
The following lemma follows from the finiteness of the set of building blocks and property (C):
\begin{lem}\label{lem:loop enough} For all $k$ there exists $n$ so
  that if $\|A_i\|>n$ then $A_i$ is at least the $k^{th}$ power of a
  combining matrix.
\end{lem}

\subsection{Column size bound}
We now begin to study the infinite matrix product $\prod A_i$. 
The first intermediate goal will be the following theorem.
\begin{thm}\label{thm:a-priori-column-bound}
  There is a number $K>0$ so that for all $n$ the ratio of the largest
  to the second smallest column of $(\prod_{i=1}^n A_i)$ is at most
  $K$.
\end{thm}
The proof requires several preliminary lemmas, beginning with the
following elementary estimate.
\begin{lem}\label{lem:improving-estimates}
  Let $a,b,c,d>0$ be four positive numbers, so that $b \leq Kd$. Then
  $$\frac{a+b}{c+d} \leq \frac{a}{c} \text{  or  } \frac{a}{c} \leq K$$
\end{lem}
\begin{proof}
  Suppose the second estimate is not satisfied, that is $a > Kc$. Then
  $$\frac{c}{a}b < \frac{c}{Kc}Kd = d$$  
  Furthermore,
  $$\frac{a+b}{c+d} = \frac{a}{c} \frac{c+\frac{c}{a}b}{c+d} <
  \frac{a}{c}$$
  This shows the estimate.
\end{proof}

\begin{lem}\label{lem:small-matrix-products}
  Suppose $B$ is a $(4\times 4)$-matrix with non-negative entries. Let
  $A$ be a matrix whose entries are bounded by $K\geq 1$. 
  \begin{enumerate}[i)]
  \item Suppose that $A$ has all positive entries. Then
    $$\frac{|C_i(BA)|}{|C_j(BA)|} \leq K$$
    for all $i,j$.
  \item Suppose that $A$ has only positive entries in at least two rows,
    and the other rows are equal to rows of the identity matrix. Then
    $$\text{either } \max_{i,j}\frac{|C_i(BA)|}{|C_j(BA)|} \leq
    \frac{|C_i(B)|}{|C_j(B)|} \text{ or } \max_{i,j}\frac{|C_i(B)|}{|C_j(B)|} \leq K$$ 
  \end{enumerate}
\end{lem}
\begin{proof}
  \begin{enumerate}[i)]
  \item The $i$-th column of $BA$ is obtained by summing the columns
    of $B$ according to the entries of the $i$-th column of $A$. In
    such a linear combination each column of $B$ appears at least
    once, and at most $K$ times by the assumption on $A$.
    Thus, the $1$-norm of any column is at least equal to $|C_1(B)| +
    \ldots + |C_4(B)|$ and at most $K(|C_1(B)| + \ldots + |C_4(B)|)$. This
    shows i).
  \item Suppose for ease of notation that rows $3$ and $4$ of $A$ are
    the ones which are not positive (if only one non-positive row
    exists then the argument works analogously). By the assumption,
    every column of $BA$ is the sum some multiples of columns $1$ and
    $2$ of $B$ and possibly one copy of column $3$ of $4$. There is a 
    permutation $\pi$ so that the entry $\pi(i)$ of row $i$ is
    positive (as the matrix $A$ is nonsingular).
    We thus have for all $i,j$:
    $$|C_j(BA)| \geq |C_{\pi(j)}(B)| + (|C_1(B)| + |C_2(B)|)$$
    and 
    $$|C_i(BA)| \leq |C_{\pi(i)}(B)| + K(|C_1(B)| + |C_2(B)|)$$
    Thus, we have
    $$\frac{|C_i(BA)|}{|C_j(BA)|} \leq \frac{|C_{\pi(i)}(B)| + K(|C_1(B)| + |C_2(B)|)
    }{|C_{\pi(j)}(B)| + (|C_1(B)| + |C_2(B)|)}$$
    By Lemma~\ref{lem:improving-estimates} (applied for
    $b=d=|C_1(B)|+|C_2(B)|$ and $a=|C_{\pi(i)}(B)|, c=|C_{\pi(j)}(B)|$) the right
    hand side is between $\frac{|C_{\pi(i)}(B)|}{|C_{\pi(j)}(B)|}$ and $K$ which
    implies the lemma.
  \end{enumerate}
\end{proof}
Recall that an entry in a combining matrix $A$ is \emph{active}, if it
is in row $i$ and column $i$ of $A$ is active. 
\begin{defin}
  If $M$ is a matrix we denote by $C_u(M)$ denote the column of $M$
  whose norm is second smallest.
\end{defin}

\begin{lem}\label{lem:large-matrix-products} 
  There is are constants $\gamma<1$, $K$  with the following property.
  Suppose $A$ is a matrix corresponding to a $n$-th power of a minimal
  depth $0$ loop for $n\geq 3$ and $B$ is a non-negative matrix. Then 
  \[ \frac{|C_{\max}(BA)|}{|C_u(BA)|}\leq \max\left\{K, \frac{|C_{max}(B)|}{|C_u(B)|}\gamma^{n-2}\right\} \]
\end{lem}
\begin{proof}
  By finiteness of the set $\mathcal{B}$ of building blocks, it
  suffices to show the statement for each matrix $A$ corresponding to
  a power of a minimal depth $0$ loop, and then let $K, \, \gamma$ be the largest of
  the resulting constants.  We argue similarly to the proof of case
  (2) of Lemma~\ref{lem:small-matrix-products}. Again, for simplicity
  of notation, assume that 1 and 2 are the active columns, 3 is
  passive and 4 is passive or idle. 

  Let $E$ be the constant from Lemma~\ref{lem:powers-of-combining}
  applied to the matrix corresponding to the depth--$0$ loop detemined
  by $A$. Fix columns $i,j$.  Let $m$ be the smallest, and $M$ be the
  size of largest of the active entries in columns $i$ and $j$. By
  Lemma~\ref{lem:powers-of-combining} we have $M \leq Em$.  Thus for
  any $i,j \neq 4$ (if column $4$ is idle) or any $i,j$ (if column $4$
  is passive) we have
  $$|C_i(BA)| \geq |C_j(B)| + m(|C_1(B)| + |C_2(B)|)$$
  and 
  $$|C_i(BA)| \leq |C_j(B)| + Em(|C_1(B)| + |C_2(B)|).$$
  Arguing exactly as in case ii) of
  Lemma~\ref{lem:small-matrix-products} one shows that quotients
  between such $i,j$ either the ratio improves, or is less than some
  fixed constant.  Moreover, because the action of combining matrices on active columns gives a positive matrix, $m$ increases exponentially with the number of loops taken and 
  therefore we have that 
  \begin{equation}\label{eq:act pass}
  \frac{|C_{i}(BA)|}{|C_j(BA)|}\leq \max\left\{K', \frac{|C_a(B)|}{|C_b(B)|}\gamma'^{n-2}\right\}
  \end{equation}
  where $i,j,a,b\neq 4$.  This establishes the claim (with $K=K'$,
  $\gamma=\gamma'$) unless $C_{\max}(BA)$ is the idle column. Indeed,
  if $C_i(BA)=C_{\max}(BA)$ where $i$ is active or passive then we
  have
 $$\frac{|C_i(BA)|}{|C_u(BA)|}\leq \underset{j \text{ active or
     passive}}{\max} \frac{|C_{i}(BA)|}{|C_j(BA)|}\leq \max\left\{K',
   \max\frac{|C_a(B)|}{|C_b(B)|}\gamma'^{n-2}\right\}$$ 
$$\leq \max\left\{K', \max\frac{|C_{\max}(B)|}{|C_u(B)|}\gamma'^{n-2}\right\}.$$

  Hence, we are done unless column $4$ of
  $A$ is idle and column $4$ of $BA$ is not the smallest column. In
  that case, it is the same as the fourth column in $B$. By splitting
  $A$ into smaller powers $A= A_0^2A_0^k$, where $A_0$ is combining
  and $k\geq 0$ we see that multiplying $B$ by $A_0^2$ already adds
  both active columns of $B$ to all but the idle column. 
  If one of the
  active columns was larger than the idle column, this implies that
  the idle column of $A$ is the smallest of $BA$. In this case we may use the bound of $K'$ from above since both $C_u$ and $C_{\max}$ are active or passive.

  Thus, we are left with the case that the idle column of $B$ is
  larger than both of the active columns of $B$.  But then for all
  $i$, in $BA_0^{i+2}$ the size of the active columns grow exponentially in $i$ and so the ratio of the idle column to $C_u$ decays exponentially. This implies the lemma if $C_{\max}(BA)$ is the idle column. Otherwise it follows from equation \ref{eq:act pass} and the fact that $\frac{|C_{\max}(B)|}{|C_u(B)|}\geq \frac{|C_i(B)|}{|C_u(B)|}$ for all $i$.
\end{proof}

\begin{cor}\label{cor:loop for fixed} 
For any $M$ there exists $N$ so that if $\|D\|_{op}\leq M$ then
$$\frac{|C_{\max}(BDA^r)|}{|C_{u}(BDA^r)|}\leq \max\left\{K,\frac{|C_{\max}(B)|}{|C_u(B)|}\right\}$$ for any $A$, a matrix given by a minimal depth zero loop and $r\geq N$. Here, $K$ does not depend on $M$.
\end{cor}
\begin{proof}
  Observe that
  $\frac{|C_{\max}(BD)|}{|C_u(BD)|} \leq \|D\|
  \frac{|C_{\max}(B)|}{|C_u(B)|}$.
  Let $\gamma$ be as in Lemma \ref{lem:large-matrix-products}. Choose
  $N$ so that $\gamma^{N-2}M\leq 1$ and the corollary follows from
  Lemma \ref{lem:large-matrix-products}.
\end{proof}

\begin{proof}[Proof of Theorem~\ref{thm:a-priori-column-bound}]
  To show the theorem, we will first define a re-grouping
  $D_k = \prod_{i=i(k)}^{i(k+1)-1}A_i$ where $i(k)$ are chosen so that
  $i(k+1)-i(k)\leq c$ and so that the products $\prod_{i=1}^n D_i$ have
  the claimed property. We will then use this to show that the theorem holds.

  \smallskip To define the $D_i$, we argue inductively. To begin with,
  note that the set of all paths
  $B_1 \to \dots \to B_k$ so that no $B_i$ is at least a third power of
  a minimal depth $0$ building block loop and so that
  $k\leq c$ is finite. Thus, the set of operator norms 
  $\|M(B_1\to \dots \to B_k)\|$ of the
  corresponding matrices is bounded -- let $M_0$ be such a bound.

  Next, let $N_1$ be the number guaranteed by Corollary~\ref{cor:loop
    for fixed} applied to the bound $M=M_0$. In particular, this
  implies that if $A$ is any matrix, and
  $B_1\to \dots \to B_k \to P^n$ where $B_1\to \dots \to B_k$ is as
  above, $P$ is a minimal depth $0$ loop and $n\geq N_1$, then
  \[\frac{|C_{\max}(AM(b_1\to \dots \to b_k \to
    P^n))|}{|C_{u}(AM(b_1\to \dots \to b_k \to P^n))|}\leq
  \max\left\{K,\frac{|C_{\max}(A)|}{|C_u(A)|}\right\}\]

  Suppose now that $N_i, M_i$ are defined. We define $M_{i+1}$ to be a
  bound for the operator norm of matrices $M(B_1\to \dots \to B_k)$
  defined by sequences with $k\leq c$ and so that no $B_i$ is a
  minimal depth $0$ loop with power at least $N_i$.
  Then, define $N_{i+1}$ to be the output of
  Corollary~\ref{cor:loop for fixed} applied to $M=M_{i+1}$.

  Finally, let $K_2$ be the output of
  Lemma~\ref{lem:small-matrix-products} applied to the bound $M_c$. We
  will now describe the grouping $D_i$ which will satisfy the column
  bound for $\max(K,K_2)$.

  Namely, suppose that $D_i$ is defined for $i\leq n$ with the desired
  properties, and $A_N$ is the final matrix appearing in the product
  defining $D_n$. By Almost Positivity, for some number $k \leq c$ the
  matrix $A_{N+1}\dots A_{N+k}$ is Almost Positive (i.e. satisfies the
  property from Almost Positivity).  There are now two
  cases:
  \begin{enumerate}
  \item $\|A_{N+1}\dots A_{N+k}\| \leq M_c$. Then we put
    $D_{n+1} = A_{N+1}\dots A_{N+l}$ and note that by
    Lemma~\ref{lem:small-matrix-products} we have
    \[\frac{|C_{\max}(D_1\dots D_nD_{n+1})|}{|C_{u}(D_1\dots D_nD_{n+1})|}\leq
    \max\left\{K_2,\frac{|C_{\max}(D_1\dots D_n)|}{|C_u(D_1\dots D_n)|}\right\}\]
  \item $\|A_{N+1}\dots A_{N+k}\| > M_c$. Then at least one of the
    $A_j, N+1\leq j\leq N+k$ corresponds to at least a $N_c$--th power of
    a minimal depth $0$ loop. Let $r$ be the minimal such $j$. We put
    $D_{n+1} = A_{N+1}\dots A_{r-1} A_r$. Note that
    $\|A_{N+1}\dots A_{r-1}\| \leq M_c$ by minimality of $r$. Thus, by
    our choice of $N_c$ we have
    \[\frac{|C_{\max}(D_1\dots D_nD_{n+1})|}{|C_{u}(D_1\dots D_nD_{n+1})|}\leq
    \max\left\{K,\frac{|C_{\max}(D_1\dots D_n)|}{|C_u(D_1\dots D_n)|}\right\}\]
  \end{enumerate}
  Thus, in any case we have
    \[\frac{|C_{\max}(D_1\dots D_nD_{n+1})|}{|C_{u}(D_1\dots D_nD_{n+1})|}\leq
    \max\left\{\max\{K,K_2\},\frac{|C_{\max}(D_1\dots
        D_n)|}{|C_u(D_1\dots D_n)|}\right\}\]
    This shows inductively the existence of the desired column size
    bound for products $\prod D_i$. The desired bound on the
    products $\prod A_i$ follows simply because by construction the
    difference between $\prod A_i$ and a suitably chosen $\prod D_j$
    has operator norm at most $M_c$ by construction.
\end{proof}

\begin{cor}\label{cor:idle smallest}
  For any $\zeta$ there is an $L$ with the following property.  Let
  $B=\prod_{i=1}^NA_i$, and suppose that $A$ is at least the $L$--th
  power of a minimal depth $0$ loop. Then the idle column of $BA^L$
  has size at most $\zeta$ times the size of any active or passive
  column of $BA^L$.
\end{cor}
\begin{proof}
  By Theorem~\ref{thm:a-priori-column-bound} the quotient in size
  between the largest and second smallest column of $B$ is at most
  $K$.  If the idle column of $B$ is not already the smallest, then
  one of the active columns is has size at least $\frac{1}{K}$ times
  the size of the idle column. After $K$ of applications of $A$ at
  this active column was added to all non-idle columns, and the idle
  column has become the smallest. Since every further application of
  $A$ increases the size of every non-idle column, the desired $L$ exists.
\end{proof}

\subsection{Angle Contraction}
\begin{lem}\label{lem:contract again}
  There are constants $K'>0, \xi>1$ with the following property. 
  Let $B$ be a matrix so that
  $\underset{1\leq i,j \leq 4}{\max}\angle(C_i(B),C_j(B)))\leq  \alpha$, 
  $\max \, \left(\frac{|C_i(B)|}{|C_u(B)|}\right)\leq D$ and let $A$
  be the matrix defined by a minimal depth $0$ loop. Then 
  $$\underset{i,j \text{active or passive}}{\max}\sin
  \angle(C_i(BA^k),C_j(BA^k))<K'\frac{\sin(\alpha)}{\xi^k}.$$
\end{lem} 
\begin{proof}

Because there are only finitely many minimal depth $0$ loops in our
building blocks, it suffices to show the lemma for a single one.

The proof is similar to Lemma \ref{lem:powers-of-combining}.  For each
combining matrix $A$ we consider its action on the active columns,
described by a matrix $\hat{A}$.  It has a Perron-Frobenius
eigenvector $[a_1,...,a_\tau]$ where $2\leq \tau\leq 4$ and all the
$a_i>0$.  It has corresponding eigenvalue $\mu$. We consider the
action of $\hat{A}$ on the invariant $(\tau-1)$-dimensional subspace
not containing the Perron-Frobenius eigenvector
$[a_1,...,a_\tau]$. Call the matrix describing this action
$\tilde{A}$. There exists $C, \nu$ where $\nu<\mu$ and
$\|\tilde{A}^k\|_{op}\leq C\nu^k.$ Thus there exists $n_A$, $\gamma>1$
so that $\frac{\mu^{k}}{C\nu^k}\geq \gamma^k$ for all $k\geq n_A$. Write
$C_i(\tilde{A}_r)=\lambda [a_1,..,a_\tau]+w$ where $\lambda$ is chosen
to be the largest possible so that
$C_i(\tilde{A}_r)-\lambda [a_1,...,a_\tau]\in
\mathbb{R}^\tau_d$.
There exists $C$, $m_A$ so that for all $k>m_A$ with
$\frac{\lambda}{|w|}>C\gamma^k$. Indeed, decompose
$C_i(\tilde{A}_r)=\lambda'[a_1,...,a_\tau]+w'$ where $w'$ is in the
subspace $\hat{A}$ acts on. By above $\frac{\lambda'}{|w'|}>\gamma^k$
for all $k>n_A$. Now letting $\xi$ be the minimal number so that
$\xi[a_1,...,a_\tau]+w'\in \mathbb{R}^\tau_+$ for all large enough $k$
we have the claim.\footnote{In particular,
  $\xi\leq \frac{|w'|}{\min a_i}$}

Consider $A_r=A^k$ where $A$ is a combining matrix and $k\geq n_A$.
Consider the active columns $C_{i_1}(B),...,C_{i_\ell}(B)$ and the
vector $\bar{a}\in \mathbb{R}^4_+$ where $\bar{a}_i$ is 0 if $C_i$ is
not active and is $a_b$ where $b$ is the corresponding column of
$\hat{A}$ if is active.  Because $\hat{A}$ captures the action of $A$
on the active columns, $C_{i_b}(A_r)$ as $\lambda\tilde{a} + w_b$
where $\lambda$ is chosen to be the largest possible with
$C_{i_b}-\lambda\tilde{a}\in \mathbb{R}^4_+$ with
$\frac{|w_b|}{\lambda}\leq C' \frac 1 {\gamma^k}$.

Now $Bw_b \in B_\Delta$ and so
$\angle( w_b, \sum a_j C_{i_j}(B))<\alpha$.  It follows from Lemma
\ref{lem:add vectors} that
$\sin(\angle (\sum a_j C_{i_j}(B), C_{i_b}(BA_r))\leq
D\gamma^{-k}\sin(\alpha)$.
Hence the lemma follows for active columns. The case of passive column
is analogous to the proof of Lemma \ref{lem:powers-of-combining}
part(iii).
  
By choosing of $K'$ large enough we absorb the various constants and
are able to remove the requirement that $k\geq m_A,\,n_A$.
\end{proof}

We next introduce some notation. If $v,w$ are line segments in
$\mathbb{R}^3$ let $|w|$ denote the length of $w$ and $\Phi_v$ denote
orthogonal projection onto the direction of $v$.  We also need a fact
from Euclidean geometry:

\begin{lem}\label{lem:sublemma}
  There exist numbers $0<c_1,c_2<1$ with the following property.
  Let $\Delta'\subset\RR^3$ be a (not necessarily
  regular) Euclidean simplex with diameter $\tau$, and let $w$ be a
  line segment in $\Delta'$ with length $|w| \geq c_1\tau$.

  Then there is a an edge $e$ of $\Delta'$ with length at least $|w|$
  so that $|\Phi_e(w)|\geq c_2|w|$.
\end{lem}
\begin{proof}  
  First observe that the diameter of a simplex is given by its longest
  side. Second, observe that if $w$ is is a line segment contained in
  $\Delta'$ then there exists a line segment $\tilde w$ contained in
  $\Delta'$, parallel to $w$, with $|\tilde{w}|\geq |w|$ and so that
  one endpoint of $\tilde{w}$ a vertex $v$ of $\Delta'$.  We thus have
  $|\tilde{w}| \geq c_1\tau$, and by choosing $c_1<1$ large enough
  there is some edge $e$ of $\Delta'$ emanating from $v$ with
  $|e| \geq |\tilde{w}|$ and whose direction is within 1 degree of the
  direction of $\tilde{w}$.  
  As a consequence, since $w$ and $\tilde{w}$ are parallel we have that 
  \[ |\Phi_e(w)|\geq \cos\left(\frac{\pi}{180}\right)|w| \]
  Thus, $c_2 = \cos(\frac{\pi}{180})c_1$ has the desired property.
\end{proof}

\begin{lem}\label{lem:distributing-angle-small} 
  For every $N$ there exists $F_1<1$ so that if $B$ is a product of
  matrices from Corollary~\ref{cor:matrix-properties},
  $A=A_{i}\cdots A_{i+k}$ is a subproduct satisfying $(P)$ or $(R)$
  with $k\leq c$ and $\|A\|_{\mathrm{op}} \leq N$. Then
  \[ \mathrm{diam}(BA\Delta) \leq F_1 \mathrm{diam}(B\Delta) . \]
\end{lem}
\begin{proof}
  Consider the (normalized) simplices $B\Delta$ and $BA \Delta$ with
  vertices $p_1,...,p_4$ and $q_1,...,q_4$ respectively. These
  correspond to the directions of the columns $C_j(B)$ and
  $C_j(BA)$. Let $\tau$ be the length of the longest side of $B\Delta$
  (which is also its diameter) and $E$ be the set of sides of
  $B\Delta$ with length at least $c_1\tau$. 

  By Lemma~\ref{lem:sublemma} it suffices to show that there is a
  constant $0<c_3<1$ so that 
  \[ |\Phi_e(\overrightarrow{q_iq_j})|<c_3|e| \] 
  for all $i,j$ and $e\in E$.

  To prove this, let $e \in E$ be arbitrary. For concreteness, assume that 
  $e = \overrightarrow{p_1p_2}$. Let $j$ be such that every entry in row $j$ 
  of $A$ is positive, and $|C_j(B)|\geq| C_u(B)|$ (such a $j$ exists as at
  least two rows of $A$ are positive). Up to exchanging $p_1$ and $p_2$ we 
  may also assume that $|\Phi_e(\overrightarrow{p_1p_j})| \geq \frac{1}{2} |e|$.

  Note that 
  \[ C_i(BA)=\sum a_{k,i}C_k(B) \]
  where $a_{k,i}\leq n$ for all $k,i$ and $a_{j,i}\geq 1$ for all
  $i$. Thus, we have
  \[ q_i = \sum a_{k,i}\frac{|C_i(B)|}{|C_i(BA)|} p_k \]
  Since $\|A\| < N$ and $|C_j(B)|\geq| C_u(B)|$, there exists some
  $\epsilon>0$ (dependent only on $N$ and the constant $K$ from
  Theorem~\ref{thm:a-priori-column-bound}) so that
  $\frac{|C_j(B)|}{|C_j(BA)|} > \epsilon$.

  Thus, we have that
  \[ \Phi_e(q_i) = \sum a_{k,i}\frac{|C_i(B)|}{|C_i(BA)|}
  \Phi_e(p_k) \]
  In this expression, the coefficient of $\Phi_e(p_j)$ is at least
  $\epsilon$, and all coefficients are bounded from above by
  $N^2$. Since
  $|\Phi_e(\overrightarrow{p_1p_j})| \geq \frac{1}{2} |e|$, there is
  therefore some $c_3$ so that all $\Phi_e(q_i)$ are within $c_3|e|$ of
  $p_2$. This shows the claim.
\end{proof}

\begin{lem}\label{lem:distributing-angle-large} 
  There exists $N>0$ and $F_2<1$ so that it $B$ is a product of matrices from
  Corollary~\ref{cor:matrix-properties}, $A=A_{i}\cdots A_{i+k}$ is a
  subproduct satisfying $(P)$ or $(R)$ with $k\leq c$ and
  $\|A\|_{\mathrm{op}} > N$.  Then
  \[ \sin(\max \angle(C_i(BAA_{i+k+1}\cdot...\cdot A_{r+k+3}), 
C_j(BAA_{i+k+1}\cdot...\cdot A_{r+k+3})))\] \[\leq F\sin(\max\angle(C_i(B),C_j(B))). \]
\end{lem}    
\begin{proof}
  The number $N$ will be given by Lemma \ref{lem:loop enough}, so that
  any $A$ with $\|A\| >N$ contains a matrix corresponding to at least
  $\ell$ powers of a minimal depth $0$ loop -- where $\ell$ will be
  chosen below.

  To begin, note that 
  \[ \sin(\max \angle(C_i(XY), C_j(XY)))
  \leq \sin(\max\angle(C_i(X),C_j(X))) . \]
  for any matrices $X,Y$ with non-negative entries (as
  $XY\Delta \subset X\Delta$). By the choice of $N$, there is an
  $s, i\leq s\leq i+k$ so that $A_s$ is the matrix defined by a
  minimal depth $0$ loop taken at least $\ell$ times. By the preceding
  remark it suffices to show that
  \[ \sin(\max \angle(C_i(BAA'), C_j(BAA'))) \leq
  F\sin(\max\angle(C_i(B),C_j(B))). \]
  where $A = A_s$ is a matrix corresponding to the $\ell$-th power of
  a combining matrix, and $A'$ is a matrix corresponding to three
  building block steps. We emphasize that this is necessarily the
  product of three $A_i$, but may be a smaller product (if some of the
  $A_i$ correspond to loops). In particular, we may assume that
  $\|A'\|_\mathrm{op} \leq L$ for some universal $L$.

  For simplicity of notation, we abbreviate
  $\beta = \sin(\max\angle(C_i(B),C_j(B)))$. By
  Lemma~\ref{lem:contract again} and our choice of $\ell$ we then have
  that
  \[ \underset{i,j \text{active or passive}}{\max}\sin
  \angle(C_i(BA),C_j(BA)) < K'\frac{\beta}{\xi^k} \]
  where active and passive refers to the columns of $A$. Let $b$
  denote the index of the bad column of $A$.

  Next, consider any column $C_k(BAA')$ and write it as a sum of columns
  of $BA$:
  \[ C_k(BAA') = a'_{kb}C_b(BA) + \sum_{i\neq b}a'_{ki}C_i(BA)\]
  
  By Isolated Idle every column of $A'$ has at least two nonzero
  entries, and thus $\sum_{i\neq b}a'_{ki} \geq 1$.

  Furthermore, the sum $\sum_i a'_{ki}$ is at most
  $\|A'\|_{op}\leq L$. Thus, we have
  \[ \frac{\|a'_{kb}C_b(BA)\|}{\|\sum_{i\neq b}a'_{ki}C_i(BA)\|} \leq
  L\frac{\|C_b(BA)\|}{\|C_i(BA)\|} \]
  for some index $i$ corresponding to an active or passive
  column. Hence by Lemma~\ref{lem:add vectors}
  \[ \sin\angle\left( C_k(BAA'), \sum_{i\neq b}a'_{ki}C_i(BA)\right)\]\[ \leq
  L\frac{\|C_b(BA)\|}{\|C_i(BA)\|}\sin\angle\left(\sum_{i\neq
    b}a'_{ki}C_i(BA), a'_{bi}C_b(BA)\right) \leq
  L\frac{\|C_b(BA)\|}{\|C_i(BA)\|}\beta\]

  On the other hand, we have
  \[ \sin\angle\left( \sum_{i\neq b}a'_{ki}C_i(BA), \sum_{i\neq
    b}a'_{mi}C_i(BA)\right)\]\[ \leq \sin\max_{r,s
    \mathrm{active,passive}}\angle C_r(BA), C_s(BA) \leq
  K'\frac{\beta}{\xi^k} \]

  Since (absolute values of) angles satisfy the triangle inequality
  and $\sin(\theta+\phi)\leq \sin(\theta)+\sin(\phi)$ for all
  $0\leq \theta,\phi$ with $\theta,\phi<\frac \pi 2$ we then have

  \[ \sin(\angle(C_k(BAA'), C_m(BAA')))\] \[<
  L\frac{\|C_b(BA)\|}{\|C_i(BA)\|}\beta
+ K'\frac{\beta}{\xi^k} 
+ L\frac{\|C_b(BA)\|}{\|C_j(BA)\|}\beta\]
\[ \leq \left(L\frac{\|C_b(BA)\|}{\|C_i(BA)\|} +
  \frac{K'}{\xi^k} +
  L\frac{\|C_b(BA)\|}{\|C_j(BA)\|} \right) \beta \]

By Corollary~\ref{cor:idle smallest}, $\ell$ may be chosen large
enough so that $L\frac{\|C_b(BA)\|}{\|C_j(BA)\|} < \frac{1}{4}$ for any
active/passive index $j$, and by enlarging $\ell$ further, we may assume that
$\frac{K'}{\xi^k} < \frac{1}{4}$ as well. Then
\[ \sin(\angle(C_k(BAA'), C_m(BAA'))) < \frac{3}{4}
\sin(\max\angle(C_i(B),C_j(B))) \]
which shows the desired outcome.
\end{proof}

The following proposition will be used to imply
convergence, continuity and unique ergodicity at an infinite depth point.
\begin{prop}\label{prop:new crit}
Let $A_i$ be as above. Then $(\prod A_i)\mathbb{R}^4_+$ converges to a
single line. 
\end{prop}

\begin{proof}
  To prove the proposition we inductively group the product
  $(\prod A_i)$ into sub-products of length at most $c$ that satisfy
  $(P)$ or $(R)$.  Depending on the norm of the next block we apply
  Lemma~\ref{lem:distributing-angle-small} or
  \ref{lem:distributing-angle-large} to conclude that either the
  diameter of the image of $\Delta$, or $\max\sin\angle$ decreases by
  a definite factor. Since one of these happens an infinite number of
  times, the proposition follows.
\end{proof}

\section{Connecting building block endpoints}\label{sec:connecting-building-blocks}
\begin{thm}
  Let $\mathcal{B}$ be a complete finite set of building blocks. and
  suppose that $S_1,S_2$ are respectively left and right endpoints for
  a building block. Then there exists a path of uniquely
  ergodic IETs connecting $S_1$ and $S_2$. 
\end{thm}

\begin{proof}
  Let $\pi=\pi(S_1)=\pi(S_2)$ be the permutation of the two points we
  want to connect.
  We will obtain the path $c$ connecting $S_1$ to $S_2$ in $\Delta_\pi$
  as the limit
  of approximating paths $c_n$. Before we define the actual paths $c_n$ though, we
  first construct a combinatorial object that will serve as a guideline
  on how to build the paths.

  We define a rooted oriented bivalent tree $\mathcal{T}$ in the following way. Each
  vertex will correspond to a building block $b \in \mathcal{B}$ and
  has two outgoing edges. The two vertices joined to $b$ are the left
  and right compatible building blocks to
  $b$ in $\mathcal{B}$. The root corresponds to the building block
  $(S_1,F,S_2)\in\mathcal{B}$, which exist by our hypothesis.
  
  Suppose $v$ is a vertex of $\mathcal{T}$. Then there is a unique path $\rho$
  which joins the root of $\mathcal{T}$ to $v$. This path $\rho$ corresponds to
  a building block sequence, which in turn defines
  a Rauzy matrix $M(\rho)$. If $v$ corresponds to the building block
  $(T_1,F,T_2)$, then we say that the triple of IETs
  $(M(\rho)T_1,M(\rho)F,M(\rho)T_2) \in \Delta_\pi$ is defined by $v$.

  Consider the set $v_0,\ldots,v_{2^n}$ of all vertices of distance
  $n$ to the root of $\mathcal{T}$, and let $(T^{(i)}_1,F^{(i)}T^{(i)}_2)$
   be the triple of
  IETs defined by $v_i$. Note that (by construction) $T_2^{(i)} = T_1^{(i+1)}$.
  We define $c_n$ as the concatenation of straight line segments in
  $\Delta_\pi$ connecting $T_1^{(i)}$ to $T_2^{(i)}$ for all $i$, parametrized
  so that each straight segment is traversed with constant speed on an
  interval of length $2^{-n}$.

  Explicitly, the first path $c_0$ is simply the straight line segment connecting
  $S_1$ to $S_2$. The path $c_1$ is the concatenation of the straight line
  connecting $S_1$ to $F$ and $F$ to $S_2$, where $(S_1,F,S_2)$ is a
  building block in $\mathcal{B}$. 
  Similarly, $c_{i+i}$ is constructed by replacing each straight line
  segment in $c_i$ by a concatenation of two line segments according
  to the building block which (after the suitable number of Rauzy
  steps) has the same endpoints.

  Note that  by definition $c_i$ and $c_{i+1}$ agree at all $t$ which
  correspond to IETs defined by vertices of distance at most $i$ to the root
  of $\mathcal{T}$.

  We call any point $c_i(t)$ which corresponds to an IET defined by
  some vertex $v$ of $\mathcal{T}$ a \emph{problematic point} of $c_i$. 
  The depth of $c_i(t)$ is defined to be the depth of the 
  building block sequence corresponding to the unique path joining
  the root of $\mathcal{T}$ to $v$.

  Explicitly, this means that the boundary points $c_1(0), c_1(1)$ are
  of depth $0$. Suppose we have assigned a depth to all the
  problematic points of $c_i$. Then, each problematic point $p$ of
  $c_{i+1}$ is adjacent to two problematic points $p_l, p_r$ of
  $c_i$. If $d_l,d_r$ are the depths of these points, then the depth
  of $p$ is $\min(d_l,d_r) + 1$.

  By the definition of building blocks every problematic point
  corresponds to a uniquely ergodic IET. 
  
  For any $t\in[0,1]$ which is not a multiple of $2^{-m}$ for some
  $m$ we now claim that the sequence $c_n(t)$ converges in
  $\Delta_n$. Namely, suppose that $t \in [k2^{-M},
  (k+1)2^{-M}]$. Then for each $i>M$ the initial part of the Rauzy expansion of $c_i(t)$
  agrees with the Rauzy expansion of the building
  block sequence corresponding to a suitable vertex $v_M$ of distance $M$ to the
  root of the tree $\mathcal{T}$. Increasing $M$ to a $N>M$ it follows from the
  construction that $v_M$ is contained in the (unique) oriented path
  joining the root of the tree $\mathcal{T}$ to $v_N$. Let $\rho$ be the
  infinite oriented path in $\mathcal{T}$ which contains all $v_i$.
  Since $t$ is not a dyadic fraction the path
  $\rho$ cannot eventually only make left or only right turns (as that
  would, e.g. for left turns, mean that $k2^{-N}\leq t \leq
  k2^{-N}+2^{-M}$ for all $M$). 

  Thus, the path $\rho$ defines a building block sequence as in
  Section~\ref{sec:infinite-depth}. Therefore,
  Proposition~\ref{prop:new crit} implies that for any
  $\epsilon$ there is a $N>0$ such that for all $i>N$ the values
  $c_i(t)$ differ by at most a distance of $\epsilon$ (since they
  share the same initial segment of building blocks). This shows that
  $c_i(t)$ converges. 

  In fact, the same argument shows that 
  $\lim_{n\to\infty}c_n(t)$ is a uniquely ergodic IET by
  Theorem~\ref{thm:ue-criterion} and that the limiting function is
  continuous at such $t$.

  It remains to show continuity at the problematic points. We show 
  continuity using the limit formulation. By construction, points on 
  $c_n$ close to a depth $L$ point follow the expansion $P$ of the 
  endpoint for longer and longer times. But in this situation,
  Lemma~\ref{lem:finite-depth-convergence} shows that the normalized
  length vectors converge, and thus implies the desired continuity.
\end{proof}

\begin{cor}\label{cor:manypaths}
  Let $T_1, T_2$ be two $4$-IETs which each become the left or right
  endpoint of a building block after a finite number of Rauzy
  steps. Then $T_1$ and $T_2$ can be joined by a path of uniquely
  ergodic IETs.
\end{cor}
\begin{proof}
  Let $K$ be large enough so that $T_1, T_2$ become endpoints after
  $K$ Rauzy steps. Now consider the triangulation $\mathcal{P}_K$ of
  $\Delta_\pi$ defined by the fail planes of Rauzy induction within
  the first $K$ steps, and
  choose endpoints of building blocks on all of the fail planes
  in $\mathcal{P}_K$ separating $T_1$ from $T_2$. Two such points
  which are not separated by a fail plane can be joined by a path as
  in the previous theorem, by applying Rauzy induction first. The
  Corollary follows by concatenating these finitely many paths.  
\end{proof}

To finish the proof of Theorem~\ref{thm:main-for-4}, we need to be able to
join IETs which are not themselves endpoints of building blocks. This
will be done with a limiting argument via the following lemmas.

\begin{lem}\label{limit path}
Let $(X,d)$ be a metric space, let $(p_i)$ be a sequence in $X$ and
$p_\infty\in X$ be so that 
\begin{enumerate}
\item $(p_i)$ converges to $p_{\infty}$
\item $p_i,p_{i+1}$ are path connected by a path $P_i \subset X$ and 
\item $\underset{i \to \infty}{\lim} \, diam (P_i)=0$.
\end{enumerate} Then $p_1$ is path connected to $p_{\infty}$.
\end{lem}
\begin{proof} Let $\psi(t)=P_j(2^jt)$ where $t\in
  [1-2^{-j+1},1-2^{-j}]$ and $\psi(1)=p_{\infty}$ Because the $P_i$
  are continuous and $P_i(1)=P_{i+1}(0)$ $\psi$ is continuous
  everywhere except possibly (1). Notice by conditions (1) and (3)
  $\underset{t \to 1^-}{\lim}\psi(t)=p_{\infty}.$ So $\psi$ extends to
  a path connecting $p_1$ to $p_{\infty}$. 
\end{proof}

\begin{lem}\label{close nice} 
Let $\mathcal{B}$ be a finite complete set of building blocks.
Let $T$ be an IET so that $R^k(T)$ is defined for all $k$. Then there
exists a sequence of IETs $(S_i)$  so that
$\pi(R^k(S_j))=\pi(R^k(T))$ for all $k<j$ and either
$R^{i-1}(S_i)=R^{i-1}(S_{i+1})$ or they are left and right endpoints
from a building block triple.
\end{lem}
\begin{proof}
  Because $\mathcal{B}$ is complete, for each
  vertex of the Rauzy diagram there is a corresponding building block where
  the right hand side takes one (forward pointing) edge and the left
  hand side takes the 
  other. Thus, we can follow the expansion of $T$ and choose the
  desired building blocks. 
\end{proof}

\begin{thm}\label{thm:main4} 
  The set of uniquely ergodic unit length $4$-IETs with
  permutation in the Rauzy class of $(4321)$ is path connected.
\end{thm}

\begin{proof}
Let $E$ by a uniquely ergodic $4$-IET with permutation $(4321)$ so
that $R^k(E)$ is defined for all $k$. 
As paths may be concatenated, it suffices to show that for each
uniquely ergodic $S$ with permutation $(4321)$ there is a path joining
it to $E$.

The first case we consider is that $S$ is a uniquely ergodic IETs that
has $R^k(S)$ defined for all $k$. Let $S_i$ be a sequence converging
to $S$ given by Lemma \ref{close nice}. By Corollary
\ref{cor:manypaths} for each $i$ there is a path, $l_i$ of uniquely
ergodic IETs that connect $S_i$ to $S_{i+1}$, and that stays in
$M(S,i)\Delta$. These satisfy condition (2) of Lemma \ref{limit
  path}. Because $S$ is uniquely ergodic and so $M(S,k)\Delta$
converges to a point, the $l_i$ satisfy condition (3) of Lemma
\ref{limit path}. For the same reason the endpoint of the $l_i$ are
converging to $S$ satisfying condition (1) of Lemma \ref{limit path}. 
Repeat the same for $E$ in place of $S$.
Thus, Lemma~\ref{limit path} yields the desired path.

We next consider the case of a uniquely ergodic IET $S$ that does not
have $R^k(S)$ defined for all $k$. There exists an interval $J$ so
that $S|_J$ has the minimal number of intervals of any induced map of
$S$. As an IET on fewer intervals, $S|_J$ has infinite Rauzy
induction. By Section \ref{sec:shadows} there exists a path in the Rauzy diagram $\mathcal{R}(4)$ whose actions on 
the relevant columns reflect the action on $S|_J$ on the appropriate
Rauzy class.  Choose once and for all finite Rauzy paths $P_i$ with
positive associated matrix for each permutation in the Rauzy diagram.

Now we choose a sequence $\hat{T}_i$ of building block endpoints converging to $T$ that share longer
and longer segments of this path and then are followed by one of the
$P_i$. We can connect the $\hat{T_i}$ and
satisfy condition (2) of Lemma \ref{limit path}. Following $T$ in Rauzy
induction leads to at least 3 columns pointing increasingly in the
direction of the length vector of $T$ and having increasing column
norm by unique ergodicity. By Corollary \ref{cor:non idoc} after the application of one of the $P$ all four columns point in roughly this direction. Also by Corollary \ref{cor:non idoc} the paths connecting $\hat{T}_i$ to $\hat{T}_{i+1}$ have diameter going to 0 as $i$ goes to infinity.\footnote{This is because at some point they have (at least) 3 large columns that almost point in $T$'s direction and all after that they follow building block matrices.}   By Corollary
\ref{cor:non idoc} the condition (1) of Lemma 
\ref{limit path} is satisfied. Because any IET in
$M(T_{n+k},n+k)\Delta$ also satisfies the assumptions of Corollary
\ref{cor:non idoc} condition (3) of Lemma \ref{limit path} is satisfied. 
\end{proof}

\section{Explicitly producing building blocks}
\label{sec:concrete}
In this section we prove
Proposition~\ref{prop:existence-of-building-blocks} by explicit
construction. The building blocks and their relations are depicted in
Figures~\ref{fig:diagram-a}, \ref{fig:diagram-b} and
\ref{fig:diagram-cd}. Unlabeled edges correspond to single Rauzy steps
(left and right corresponds to type a and b Rauzy induction. Labeled
edges of the graphs correspond to several common Rauzy steps (the
Rauzy steps corresponding to each path can be found in
Appendix~\ref{sec:composite-paths}). Edges ending in capital letters
lead to the corresponding blocks in a different diagram.
\begin{figure}
  \centering
  \includegraphics[width=\textwidth]{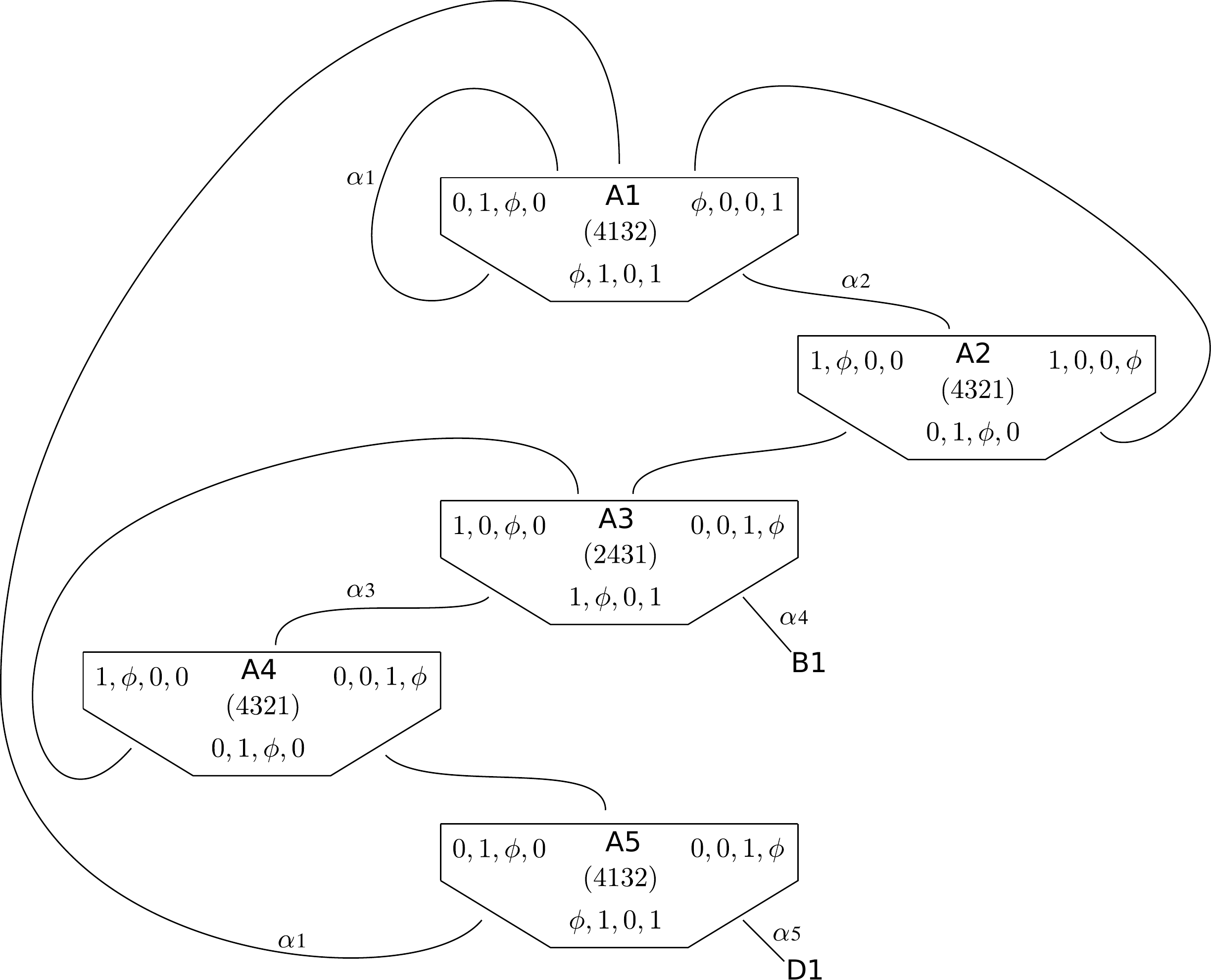}
  \caption{Diagram A}
  \label{fig:diagram-a}
\end{figure}
\begin{figure}
  \centering
  \includegraphics[width=\textwidth]{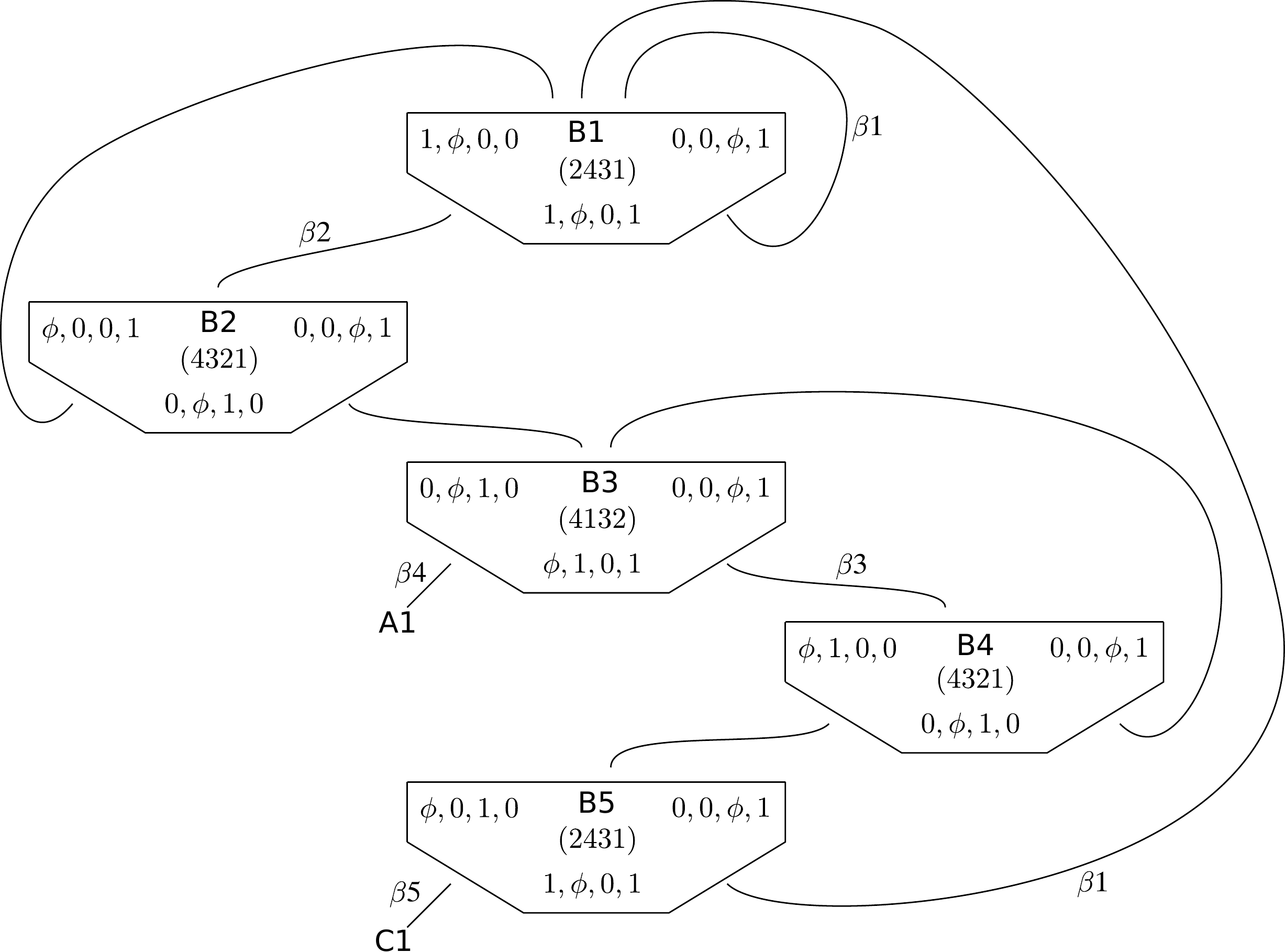}
  \caption{Diagram B}
  \label{fig:diagram-b}
\end{figure}
\begin{figure}
  \centering
  \includegraphics[width=\textwidth]{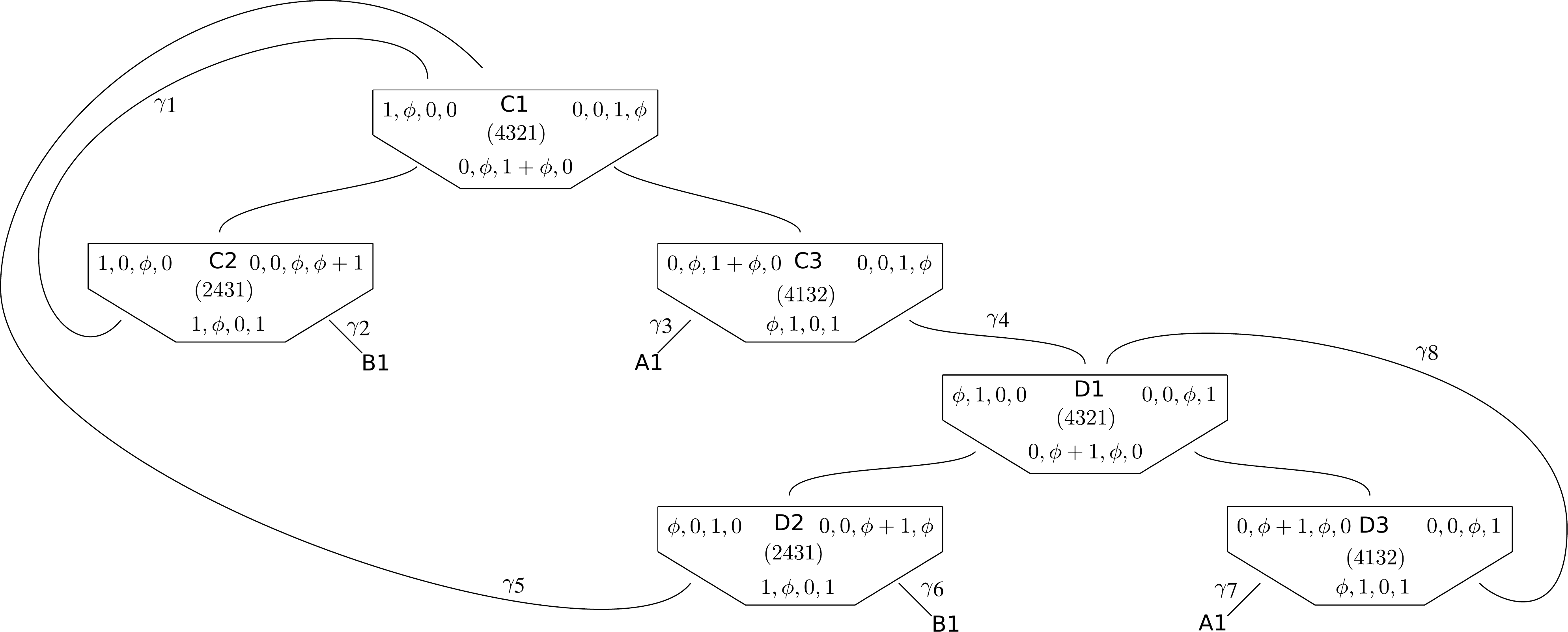}
  \caption{Diagram CD}
  \label{fig:diagram-cd}
\end{figure}
\begin{figure}
  \centering
  \includegraphics[width=\textwidth]{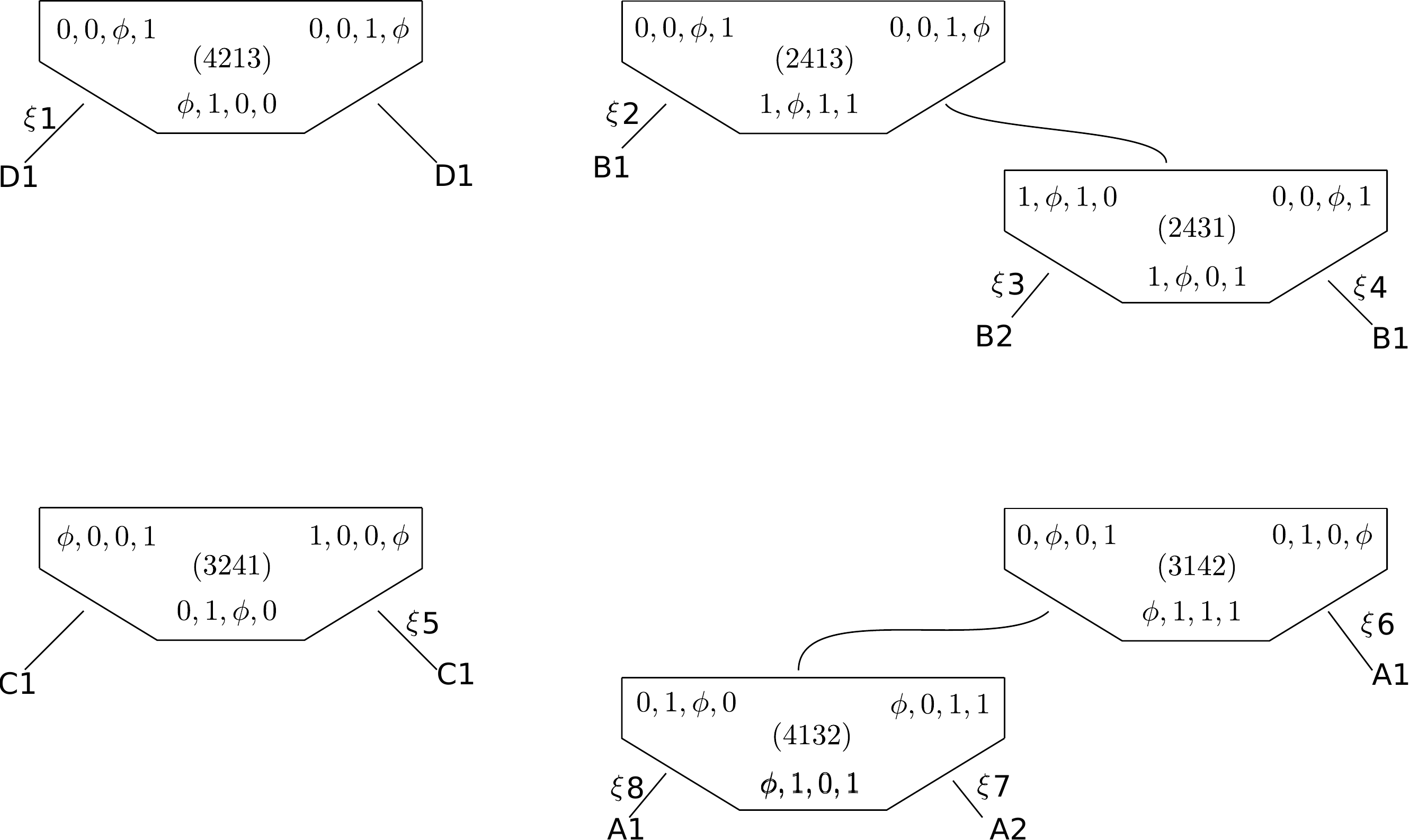}
  \caption{Diagram EFGH}
  \label{fig:diagram-efgh}
\end{figure}

By simply observing the Rauzy diagram and Figures~\ref{fig:diagram-a}
to~\ref{fig:diagram-efgh} it is immediate that the set of building
blocks satisfies \textbf{Transitivity} and \textbf{Completeness}
claimed in Proposition~\ref{prop:existence-of-building-blocks}.

To check \textbf{Combining Loops} and \textbf{Isolated Idle}, one checks in the diagrams
that the following paths are the only periodic paths of depth $0$, and
that their matrices have the desired properties.
\begin{enumerate}
\item $A1 \to A1$ corresponds to $M_{\alpha1}$.
$$M_{\alpha 1} = 
\begin{pmatrix}
 1 & 0 & 0 & 0 \\
 0 & 1 & 1 & 1 \\
 0 & 1 & 2 & 0 \\
 0 & 0 & 0 & 1
\end{pmatrix}$$
Columns 2 and 3 are active, 4 is passive. Column $1$ is idle. The
matrices which can follow without repeating this loop are
$\alpha_2M^{(4321)}_a\alpha_3, \alpha_2M^{(4321)}_a\alpha_4$,
$\alpha_2M^{(4321)}_b\alpha_1$ or $\alpha_2M^{(4321)}_b\alpha_2$, all
of which have at least two nonzero entries in each column. Thus,
Isolated Idle holds for this loop.
\item $A2 \to A1 \to A2$ corresponds to $M^{(4321)}_bM_{\alpha2}$
$$\begin{pmatrix}
 1 & 1 & 1 & 1 \\
 0 & 1 & 0 & 0 \\
 0 & 0 & 1 & 0 \\
 1 & 2 & 2 & 2
\end{pmatrix}$$
Columns 1 and 4 are active, 2 and 3 are passive.
\item $A3 \to A4 \to A3$ corresponds to
  $M_{\alpha3}M^{(4321)}_a$
$$\begin{pmatrix}
 1 & 1 & 1 & 1 \\
 0 & 1 & 0 & 0 \\
 1 & 1 & 2 & 0 \\
 0 & 0 & 0 & 1
\end{pmatrix}$$
Columns 1 and 3 are active, 2 and 4 are passive.
\item $B1 \to B1$ corresponds to $M_{\beta1}$
$$\begin{pmatrix}
 1 & 0 & 0 & 0 \\
 0 & 1 & 0 & 0 \\
 0 & 0 & 2 & 1 \\
 1 & 0 & 1 & 1
\end{pmatrix}$$
Columns 3 and 4 are active, 1 is passive. Column 2 is idle. The
matrices which can follow are $M_{\beta2}M^{(4321)}_a\beta_1$,
$M_{\beta2}M^{(4321)}_a\beta_2$, $M_{\beta2}M^{(4321)}_bM_{\beta3}$ or
$M_{\beta2}M^{(4321)}_bM_{\beta4}$, all of which have at least two
nonzero entries in each column. Thus, Isolated Idle holds.
\item $B1 \to B2 \to B1$ corresponds to $M_{\beta2}M^{(4321)}_a$
$$\begin{pmatrix}
 1 & 1 & 1 & 1 \\
 1 & 2 & 1 & 1 \\
 0 & 0 & 1 & 0 \\
 0 & 0 & 0 & 1
\end{pmatrix}$$
Columns 1 and 3 are active, columns 3 and 4 are passive.
\item $B3 \to B4 \to B3$ corresponds to $M_{\beta3}M^{(4321)}_b$
$$\begin{pmatrix}
 1 & 0 & 0 & 0 \\
 0 & 1 & 0 & 0 \\
 1 & 0 & 2 & 1 \\
 1 & 1 & 1 & 1
\end{pmatrix}$$
Columns 3 and 4 are active, columns 1 and 2 are passive.
\item $C1 \to C2 \to C1$ corresponds to $M^{(4321)}_aM_{\gamma1}$
$$\begin{pmatrix}
 1 & 1 & 1 & 1 \\
 1 & 2 & 0 & 0 \\
 0 & 0 & 1 & 0 \\
 0 & 0 & 0 & 1
\end{pmatrix}$$
Columns 1 and 2 are active, columns 3 and 4 are passive.
\item $D1 \to D3 \to D1$ corresponds to $M^{(4321)}_bM_{\gamma8}$
$$\begin{pmatrix}
 1 & 0 & 0 & 0 \\
 0 & 1 & 0 & 0 \\
 0 & 0 & 2 & 1 \\
 1 & 1 & 1 & 1
\end{pmatrix}$$
Columns 3 and 4 are active, columns 1 and 2 are passive.
\end{enumerate}

It remains to check \textbf{Almost Positivity}. We need to show that 
paths that increase depth enough yield almost positive matrices.

We begin by noting two reductions that describe why this is a finite task:
\begin{itemize}
\item It suffices to consider paths starting at
  $A1, A3, B1, B3, D1, C1$, as all paths we will construct will end at a such a
  point. Initial segments leading to one of these points have
  uniformly small length and do not affect the argument.
\item All minimal depth $0$ loops have matrices with
  positive entries on the diagonal. Since multiplying by such a matrix
  preserves the property of being positive, any path that yields a
  positive matrix will still have this property if any number of loops
  is inserted.
\item If a path without loops yields an almost positive, yet not
  positive matrix, there is a finite number of modifications (add
  loops of length $1$) that need to be considered. If all such
  modificiations yield positive matrices, we are done.
\end{itemize}
The desired property now follows by a lengthy case-by-case check. We
only list the cases for paths starting at $A1$ in details. All other
cases are checked in a completely analogous manner\footnote{which can
  also be automated and done with computer algebra systems -- see the
  second author's webpage for a SAGE script doing this task.}

The following observation helps to reduce the number of cases:
\begin{obs}
  Say that a $(4\times 4)$--matrix $A$ is \emph{weakly positive} if
  every entry is non-negative and there is at most one zero entry.  

  If $A$ is a weakly positive matrix and $B$ is Almost Positive or
  weakly positive, then $AB$ has positive entries.
\end{obs}
Thus, it suffices to show a slightly weaker version of Almost
Positivity: \emph{any path that increases depth enough yields a matrix
  that is Almost Positive, or is weakly positive.}
\begin{description}
\item[Paths that begin with the loop $\alpha_1$ at A1] There are four
  possibilities that can follow the $\alpha_1$-loop: 
  \begin{enumerate}
  \item The path $A1\to A2\to A3\to B1$, corresponding to the matrix
    $$\alpha_1\alpha_2 M_a^{(4321)} \alpha_4 = \begin{pmatrix}
      2 & 2 & 1 & 2 \\
      2 & 1 & 2 & 4 \\
      2 & 0 & 1 & 3 \\
      1 & 1 & 1 & 2
    \end{pmatrix}$$
    which is weakly positive.
  \item The path $A1\to A2\to A3\to A4\to A5\to D1$, corresponding to the matrix
    $$\alpha_1\alpha_2 M_a^{(4321)} \alpha_3 M_b^{(4321)} \alpha_5 = \begin{pmatrix}
      4 & 5 & 4 & 2 \\
      3 & 5 & 3 & 1 \\
      1 & 2 & 2 & 0 \\
      2 & 3 & 2 & 1
    \end{pmatrix}$$
    which is weakly positive.
  \item The loop $A1\to A2\to A3\to A4\to A5\to A1$, which immediately
    leads to a positive matrix.
  \item The loop $A1\to A2\to A1$, leading to 
    $$\alpha_1\alpha_2 M_b^{(4321)} = \begin{pmatrix}
      2 & 1 & 1 & 1 \\
      1 & 2 & 2 & 1 \\
      0 & 1 & 2 & 0 \\
      1 & 1 & 1 & 1
    \end{pmatrix}$$
    If this is followed by a path beginning with $A1\to A3$, then one
    of the previous cases implies that we will reach a positive
    matrix. The only other possibility is to alternate with the
    $A1\to A1$ loop. But, the square of the above matrix is positive,
    and thus this case also yields a positive matrix.
  \end{enumerate}
  As a consequence we record: \emph{Any path that begins with the
    $\alpha_1$--loop and increases depth will yield a positive
    matrix.} In particular, in the consequent analysis we never have
  to consider the case that the $\alpha_1$--loop is inserted in the
  middle of some path -- as then the resulting matrix will be positive
  upon increasing depth enough.
\item[Paths that begin with the loop $A1 \to A2 \to A1$] 
  \begin{enumerate}
  \item If the loop is followed by the $\alpha_1$--loop, then the
    above show that the result will turn positive.
  \item If we follow by $A1\to A3$ the resulting matrix is
    $$\alpha_2 M_b^{(4321)} \alpha_2 M_a^{(4321)} = \begin{pmatrix}
      2 & 5 & 4 & 4 \\
      0 & 0 & 1 & 0 \\
      0 & 0 & 0 & 1 \\
      1 & 3 & 3 & 3
    \end{pmatrix}$$
    which is Almost Positive. 
  \end{enumerate}
  In consequence, \emph{Any path that begins with the
    $(A1\to A2\to A1)$--loop and increases depth will yield an Almost
    Positive matrix.}
\item[Paths that begin with $A1 \to A2 \to A3 \to B1$] 
  Here, we start with the matrix
    $$\alpha_2 M_a^{(4321)} \alpha_4 = \begin{pmatrix}
      2 & 2 & 1 & 2 \\
      0 & 0 & 1 & 1 \\
      1 & 0 & 0 & 1 \\
      1 & 1 & 1 & 2
    \end{pmatrix}$$
  which is followed by $B1 \to B3$, leading to
    $$\alpha_2 M_a^{(4321)} \alpha_4 \beta_2 M_b^{(4321)} = \begin{pmatrix}
      6 & 5 & 6 & 2 \\
      0 & 1 & 1 & 0 \\
      1 & 1 & 2 & 0 \\
      3 & 3 & 4 & 1
    \end{pmatrix}$$
  From here, there are several possibilities:
  \begin{enumerate}
  \item Returning to $A1$, leading to 
    $$\alpha_2 M_a^{(4321)} \alpha_4 \beta_2 M_b^{(4321)} \beta_4 = \begin{pmatrix}
      6 & 5 & 11 & 7 \\
      0 & 1 & 2 & 1 \\
      1 & 1 & 3 & 1 \\
      3 & 3 & 7 & 4
    \end{pmatrix}$$
    which is weakly positive.
  \item Following $B3\to B5 \to B1$, leading to a positive matrix.
  \item Following $B3 \to B5 \to C1$, leading to a weakly positive matrix.
  \end{enumerate}
\item[Paths that begin with $A1 \to A2 \to A3 \to A5 \to D1$] 
  This corresponds to the matrix
    $$\alpha_2 M_a^{(4321)} \alpha_3 M_b^{(4321)} \alpha_5 = \begin{pmatrix}
      4 & 5 & 4 & 2 \\
      1 & 2 & 0 & 0 \\
      0 & 0 & 1 & 0 \\
      2 & 3 & 2 & 1
    \end{pmatrix}$$
    There are several possibilities:
    \begin{enumerate}
    \item following with $D1 \to D2 \to B1$ yields a weakly positive matrix.
    \item following with $D1 \to D2 \to C1$ yields a Almost Positive matrix. 
    \item following with $D1\to D3\to D1 \to D2 \to C1$ (inserting the
      $D1$ loop into the previous case) leads to 
    $$\alpha_2 M_a^{(4321)} \alpha_3 M_b^{(4321)} \alpha_5 ( M_b^{(4321)}\gamma_8) M_a^{(4321)}\gamma_5 = \begin{pmatrix}
      6 & 13 & 16 & 12 \\
      1 & 3 & 1 & 1 \\
      0 & 0 & 2 & 0 \\
      3 & 7 & 8 & 6
    \end{pmatrix}$$
    Following with $C1 \to C3$ leads to a weakly positive
    matrix. Following with $C1\to C2\to B1$ leads to a positive
    matrix. Since both of these lead to positive matrices, we are done. 
    \end{enumerate}
  \item[Paths that begin with $A1 \to A2 \to A3 \to A5 \to A1$] Note
    that the matrix describing $A5 \to A1$ in fact is the same as
    $A1 \to A1$, and thus this follows from the first case.
\end{description}

\section{Extending to $n$-IETs}
\label{sec:n-iets}
The goal of this section is to prove
\begin{thm}\label{thm:gen case}Let $\pi$ be a non-degenerate permutation for $n$-IETs
  where $n\geq 4$. The set of
  uniquely ergodic $n$-IETs with permutation $\pi$ is path connected. 
\end{thm}
The main tool in the proof of this is the following object.
\begin{defin}
  An IET $S$ is called a \emph{secret $4$-IET at level $k$} if there
  exists $M(T,r)$ a matrix of Rauzy induction with $r\leq k$, $v$ a
  non-negative vector with at most 4 non-zero entries so that
  $L(S)=M(T,r)v$.
\end{defin}
In other words, a secret $4$-IET at level $0$ is an $n$-IET which
has at most $4$ intervals of nonzero length. A general secret $4$-IET
is one that is of this form after applying some number of Rauzy steps.

Next we describe how $\pi$ acts on subsets of $\{1,..,d\}$. 
Let $(p_1,\ldots,p_4)$ be a $4$-tuple of numbers, labeled so that
$p_1<p_2<p_3<p_4$. Applying $\pi$ to the $p_i$ maps them to numbers
$\pi(p_i)$ which may not be in order. Let $\pi|(p_1,p_2,p_3,p_4)$ be
the permutation on the numbers $\{1,2,3,4\}$ which describes how the
order of $p_i$ is rearranged by $\pi$. We call $\pi|(p_1,p_2,p_3,p_4)$
the \emph{permutation $\pi$ restricted to $(p_1,\ldots,p_4)$}.

To give a simple example, consider $\pi=(3142)$: then $\pi|(3,4) = (12)$
and $\pi|(1,3) = (21)$.

We say that a permutation $\pi$ \emph{acts as a rotation} on a pair
$(p_1,p_2)$ of symbols if $\pi|(p_1,p_2) = (21)$. A permutation $\pi$
is \emph{irreducible} if it does not preserve a set of the form
$\{1,\ldots,k\}$ for $k<n$.
\begin{defin}
  Let $\pi$ be a permutation on $n$ symbols and $(p_1,\ldots,p_4)$,
  $(q_1,\ldots,q_4)$ be two $4$-tuples of entries so that $\pi$ restricts on
  both as a permutation in the Rauzy class of $(4321)$. We call the tuples
  \emph{accessible} if there exists index pairs $(r_1,r_2)$, $(s_1,s_2)$ so that $\pi$
  acts on $(p_{r_1},p_{r_2},q_{s_1},q_{s_2})$ irreducibly, 
   $\pi$ acts as a rotation on $(p_{r_1},p_{r_2})$ and
  $(q_{s_1},q_{s_2})$ and we do not have $\pi|(p_{r_1},p_{r_2},q_{s_1},q_{s_2})=(4231)$.
\end{defin}
We require that $\pi|(p_{r_1},p_{r_2},q_{s_1},q_{s_2})\neq(4231)$
 for technical reasons relating to the proof of the next lemma which does not work for this permutation. 
In particular, most of the analysis is for the case where $\pi|(p_{r_1},p_{r_2},q_{s_1},q_{s_2})$ acts as a 3-IET with a pair of $p_i,q_j$ behaving as one interval. Our argument does not work when this interval is the second one, which is the case of $(4231)$.
\begin{lem}\label{lem:joining-accessible}
Fix a non-degenerate permutation $\pi$ on $d$ symbols and suppose that 
$T,S$ are uniquely ergodic IETs with permutation $\pi$.
Further assume 
that the lengths vectors of $T$ and $S$ are $0$ off of
$(p_1,\ldots,p_4)$ and $(q_1,\ldots,q_4)$ respectively (i.e. that they
are secret $4$-IETs). 

If $(p_1,\ldots,p_4)$ and $(q_1,\ldots,q_4)$ are accessible then there exists a
path of uniquely ergodic IETs connecting $T$ to $S$. 
\end{lem}
\begin{proof} 
  If $\pi|(p_{r_1},p_{r_2},q_{s_1},q_{s_2})\in \mathcal{R}(4321)$ or
  is a rotation (for example if
  $\pi|(p_{r_1},p_{r_2},q_{s_1},q_{s_2})=(3412)$) then we can choose
  $T_1$ and $S_1$ secret $2$-IETs at level zero on $(p_{r_1},p_{r_2})$
  and $(q_{s_1}, q_{s_2})$ which are irrational rotations by the same
  number. The set of all $n$-IETs whose length vectors are $0$ off of
  the entries $(p_1,\ldots,p_4)$ is a copy of $4$-IET space. By
  Theorem~\ref{thm:main-for-4}, $T$ can therefore be connected to
  $T_1$ by a continuous path. The analogous statement is follows for
  $S, \, S_1$. Applying Theorem~\ref{thm:main-for-4} once more to
  connect $T_1$ to $S_1$ we obtain the desired path from $T$ to $S$ by
  concatenation.

\smallskip
If $\pi|(p_{r_1},p_{r_2},q_{s_1},q_{s_2})$ acts as a rotation then
$(r_i,s_j)$ and $(r_k,s_l)$ can be treated as one interval and we can
change weight between them, without changing the self-map of the
interval that the IET defines. This defines a path of uniquely ergodic IETs
connecting $T_1,S_1$.  
 
\smallskip
If $\pi|(p_{r_1},p_{r_2},q_{s_1},q_{s_2})$ acts as an element in
$\mathcal{R}(321)$ then one of the three intervals is split in two. 
Choose $T_1,S_1$ to be secret 2-IETs which are rotations by the same irrational number.
 These are path connected to $T$ and $S$ respectively by Theorem~\ref{thm:main-for-4}.
Recall from the comments in the Introduction following Theorem \ref{thm:main} that 3-IETs can be thought of as the
induced map of rotation by ${\frac{1-L_1}{1+L_2}}$ on $[0,\frac 1 {1+L_2})$
 and they are uniquely ergodic if and
only if $\frac{1-L_1}{1+L_2}$ is irrational. 
Now we connect $T_1$ and $S_1$ in the set of IETs with length zero except on $p_{r_1},p_{r_2},q_{s_1},q_{s_2}$ 
keeping $\frac{1-L_1}{1+L_2}\notin \mathbb{Q}$ constant, which keeps the 3-IETs all uniquely ergodic. We can do this because $T_1$ and $S_1$ are rotations by the same irrational number.

To see that this is doable,
consider $\alpha \in [0,1]\setminus \mathbb{Q}$ and the line segment $$H_\alpha=\Delta_3\cap \{(x_1,x_2,x_3):\frac{x_2+x_3}{x_1+2x_2+x_3}=\alpha\}.$$ Observe that $(1-\alpha,0,\alpha)\in H_\alpha$. 
If $\alpha<\frac 1 2 $ then $(\frac{1-2\alpha}{1-\alpha},\frac\alpha {1-\alpha},0)$ is in this set. If $\alpha>\frac 1 2$ then $(0, \frac{1-\alpha}\alpha,\frac{2\alpha-1}\alpha)\in H_\alpha$.\footnote{The fact that if $\alpha \notin \mathbb{Q}$ we only get one of $(1-\alpha,\alpha,0)$ or $(0, 1-\alpha, \alpha)$ in $H_\alpha$ is why we rule out the case of $(4231)$.}
If we have $(p_1,q_1,p_2,q_2)\to (p_2,q_2,q_1,p_1)$ (where this notation means that $p_1<q_1<p_2<q_2$ and $\pi$ restricted to these intervals puts $p_2$ first, then $q_2$ then $q_1$ then $p_1$) choose $\alpha>\frac 12$.  Start at $(1-\alpha,0,\alpha,0)$ and move in a straight line to $(1-\alpha,0,0,\alpha)$. From here move along the path $(a_t,b_t,0,c_t)$ where $(a_t,b_t,c_t)$ parametrizes the line between $(1-\alpha,0,\alpha)$ and $(0, \frac{1-\alpha}\alpha,\frac{2\alpha-1}\alpha)$ in $H_\alpha$. 
The case of $(p_1,q_1,q_2,p_2) \to(q_2,p_2,q_1,p_1)$ is similar.  
The cases of $(p_1,q_1,p_2,q_2) \to (q_2,p_2,p_1,q_1)$ and $(p_1,q_1,q_2,p_2)$ are also similar (if $L(T)=(a,b,c)$ then $L(T^{-1})=(c,b,a)$). 
\end{proof}

The proof of Theorem~\ref{thm:gen case} relies mainly on the following
combinatorial statement:
\begin{prop}\label{prop:accessible-chain}
Let $\pi$ be irreducible. Any two $4$-tuples of entries so that $\pi$
acts on both as a permutation in the Rauzy class of $(4321)$ can be
linked by an accessible chain. That means there is a sequence of
$4$-tuples, so that consecutive ones are accessible, which begins and
ends with the given $4$-tuples.
\end{prop} 

\begin{lem} \label{lem:nonaccessible-order} If $(p_1,\ldots,p_4)$,
  $(q_1,\ldots,q_4)$ are not an accessible pair then up to swapping
  $(p_1,\ldots,p_4)$ with $(q_1,\ldots,q_4)$ we have $p_i<q_j$,
  $\pi(p_i)<\pi(q_j)$ for all $i,j$.
\end{lem}
\begin{proof} 
  We assume $p_1<p_2<p_3<p_4$ and $q_1<...<q_4$ and $p_1<q_1$. We show
  the lemma by contradiction by considering numerous cases.
  
  We first handle the case that $p_1<q_1<p_4$.

  \smallskip
  \textbf{Subcase a:} If $\pi(p_i)<\pi(q_j)$ for all $i,j$ then either
  $$p_1<q_1<p_i<q_j \text{ or }p_1<q_1<q_j<p_i$$ where $\pi(p_i)<\pi(p_1)$ and $\pi(q_j)<\pi(q_1)$ $(p_i$ and $q_j$ both exist by the irreducibility of $\pi$). So we obtain the permutations $(3142)$ or $(4132)$ which are both in $\mathcal{R}(4321)$, and hence the tuples are accessible.

  \smallskip \textbf{Subcase b:} If $\pi(p_i)>\pi(q_j)$ for all $i,j$ then we can either pick $p_\ell<q_1<q_j<p_i$ where $\pi(q_j)<\pi(q_1)<\pi(p_\ell)<\pi(p_j)$ or $p_1<p_i<q_j<q_\ell$ and $\pi(q_\ell)<\pi(q_j)<\pi(p_i)<\pi(p_1)$. These correspond to permutations in $\mathcal{R}(4321)$.

  \smallskip \textbf{Subcase c:} So we may now treat the case $p_1<q_1<p_i$ and $\pi(p_\ell)<\pi(q_j)<\pi(p_k)$. If $j$ can be chosen to be $1$ there exists $q_a$ so that $\pi(q_a)<\pi(q_1)$. Also by the irreducibility of $\pi$ on $p_1,p_2,p_3,p_4$ we may choose $p_b$ so that $p_b<q_1$ and $\pi(p_b)>\pi(q_1)$ or $p_b>q_1$ and $\pi(p_b)<\pi(q_1)$. (Otherwise if $p_i<q_q<p_j$ then $\pi(p_i)<\pi(q_1)<\pi(p_j)$.)  If both occur 
 our permutation can by chosen to be $(4321)$ or $(3421).$ If there is only $p_b>q_1 $ so that $\pi(p_b)<\pi(q_1)$ then by the irreducibility of $\pi$ on $p_1,...,p_4$ there exists $q_1<p_c<p_b$ so that $\pi(p_b)<q_1<\pi(p_c)$. So we obtain a permutation corresponding to $(3214)$, $(3241)$ by choosing $q_1,q_a,p_c,p_b$ to be our symbols and treating the position of $q_a$ relative to $p_b,p_c$ we obtain the permutation $(4213),(4312)$ or $(3421)$. 
 If $j$ can not be chosen to be $1$ then we may choose our permutation to be $(4312)$.

  The next case is that $\pi(p_k)<\pi(q_j)<\pi(p_L)$ and is proved similarly.
  
  The last case is that $p_4<q_1$ and $\pi(p_i)>\pi(q_j)$ for all
  $i,j$. We pick $p_k<p_L$ and $q_i<q_j$ so that $\pi(p_L)<\pi(p_k)$
  and $\pi(q_j)<\pi(q_i)$. $\pi|(p_L,p_k,q_i,q_j)= (4321)$ and we are
  done. 
\end{proof}

We now begin the proof of Proposition~\ref{prop:accessible-chain}. If
the given $4$-tuples are not accessible, then we may assume (by
Lemma~\ref{lem:nonaccessible-order}) that $q_j>p_i$ and
$\pi(q_{j})>\pi(p_{i})$ for all $i,j$, and we do so from now on.

The proof proceeds by induction. To do that, we define the
\emph{distance} of two $4$-tuples $(q_1,\ldots,q_4),(p_1,\ldots,p_4)$
to be 
$$(\min_i\{q_i\}-\max_i\{p_i\},\min_i\{\pi(q_i)\}-\max_i\{\pi(p_i)\})$$ 
ordered by lexicographic order. The distance is at least $(1,1)$ by
our assumption and Lemma~\ref{lem:nonaccessible-order}.

If the distance is larger than that, then the following lemma allows
to decrease the distance.
\begin{lem} \label{lem:decr dist} If $P=(p_1,\ldots,p_4)$ and
  $Q=(q_1,\ldots,q_4)$ are not an accessible pair then there 
  exists $A=(a_1,\ldots,a_4)$ so that $A$ and $Q$ are accessible and
  the distance from $P$ to $A$ is strictly less than the distance from
  $P$ to $Q$. 
\end{lem}
\begin{proof} The proof is done in 2 cases. 

\textbf{Case 1}: In this case we assume there exists $r<\min\{q_i\}$ so that $\pi(r)>\pi(q_i)$ for some $i$.  
 Now if $\pi(r)<\pi(q_j)$ if and only if $\pi(q_1)<\pi(q_j)$ for $j\neq 1$, replace
 $q_1$  with $r$ and $\pi|(r,q_2,q_3,q_4)=\pi|(q_1,q_2,q_3,q_4)$. Call this Case 1'.
 Otherwise, there are two sub-cases. If $\pi(r)<\pi(q_1)$ then by the assumption of Case 1 there
 exists $q_i$ so that $\pi(q_i)<\pi(r)$. Since we are not in Case 1' there exists $q_j$ so that
 $\pi(r)<\pi(q_j)<\pi(q_1).$
 $\pi|(r,q_1,q_i,q_j)\in\{(3142),(4132)\}$, both of which are in the
 $\mathcal{R}(4321)$.  
 If $\pi(r)>\pi(q_1)$ because we are not in Case 1',
 there exists $q_i$ so that $\pi(q_1)<\pi(q_i)<\pi(r)$ and there
 exists $q_j$ so that $q_1<q_j<q_i$ so that $\pi|(r,q_1,q_i,q_j)\in
 \{(2413) ,(3241),(2431)\}$. (This is because the string $12$ never appears as a
 consecutive pair in a permutation in the Rauzy class of $(4321)$, since in this case the pair of intervals $I_1\cup I_2$ could be treated as one interval.)

 \textbf{Case 2}: If we are not in Case 1, then Lemma \ref{lem:switch} implies
 we can have an accessible four-tuple which has $\min\{q_i\}$ and  
 the absence of Case 1 implies $\min\{\pi(q_i)\}>\min\{q_i\}$. So some
 symbol greater than $\min\{q_i\}$ is sent before 
 $\min\{\pi(q_i)\}$ (by counting). The proof is now finished similar to Case 1. 
\end{proof}

Finally, the case of distance of $(1,1)$ is handled by the next lemma.
\begin{lem} \label{lem:switch} For any $i$ either there exists $r<i$
  so that $\pi(r)>\pi(i)$ or there exists $r>i$ so that
  $\pi(r)<\pi(i)$.  
\end{lem}
\begin{proof} If $i=\pi(i)$ this simply follows by the irreducibility of $\pi$. 

  If $i\neq \pi(i)$ then this by a counting
  argument. Indeed if $i<\pi(i)$ then there exists $r>i$ so that
  $\pi(r)<\pi(i)$ because $\pi$ is a bijection on a finite set and
  there is no injection from a set of large cardinality to a set of
  smaller cardinality. 
\end{proof}

Let us see that this lemma implies
Proposition~\ref{prop:accessible-chain} with distance $(1,r)$ for any $r$. Let
$i=q_1$. So let us assume that there is $j<i$ so that $\pi(j)>\pi(i).$
Now similarly to Case 1 of Lemma~\ref{lem:decr dist} we can replace
one of the $q_k$ with $j$ and obtain a new permutation on 4 symbols in
the Rauzy class of $4$-IETs. Now since $p_4=q_1-1$ we have that that
the new pair violates Lemma~\ref{lem:nonaccessible-order}. The other
possibility is similar. 

This finishes the proof of Proposition~\ref{prop:accessible-chain}.

\begin{cor}\label{cor:zero} 
  The set of secret $4$-IETs of level $k$ is path connected for any $k$.
\end{cor}
\begin{proof}
 For secret $4$-IETs at level $0$ this follows by Lemma~\ref{lem:joining-accessible} and
 Proposition~\ref{prop:accessible-chain}. Otherwise, argue as in the proof of
 Corollary~\ref{cor:manypaths} with the triangulation $\mathcal{P}_N$.
\end{proof}

The final ingredient to the proof of Theorem~\ref{thm:gen case} is the
following lemma. 
\begin{lem}\label{lem:well connected} 
Let $T$ be an IET so that $R^k(T)$ is defined for all $k$. Let
$S_1,S_2$ be secret 4-IETs at depth $r$ which are not secret at depth $k$. Let
$S_1,S_2 \in M(T,r)\Delta$. Further assume that $M(R^k(T),r-k)$ is a
positive matrix. Then $S_1,S_2$ are path connected by uniquely ergodic
IETs in $M(T,k)\Delta$.  
\end{lem}
The idea of this proof is that we build paths connecting $R^kS_1$ to $R^kS_2$ and the path we are interested in is the (projective) image of this path under $M(T,k)=M(S_i,k)$.
\begin{proof} Because
  $M(R^kT,r-k)$ is positive, it follows that if $S \in M(T,r)\Delta$
  is uniquely ergodic and $R^r(S)$ has two nonzero entries then it is
  not a secret $4$-IET at level $k$. 
  So it suffices to show that in the simplex the secret $4$-IETs at level
  at most $L$ are path connected for all $L$. This follows from
  Corollary~\ref{cor:zero}. Indeed, let $S_i'=R^k(S_i)$ and $P$ be the path given by 
  Corollary~\ref{cor:zero} connecting $S_1'$ and $S_2'$. $M(T,k)P$ is the path connecting $S_1$ and $S_2$ contained in $M(T,k)\Delta$.
\end{proof} 

\begin{proof}[Proof of Theorem \ref{thm:gen case}]
If $S,E$ have $R^k$ defined for all $k$ then the theorem follows
analogously to the proof of Theorem~\ref{thm:main4}
before via Lemma~\ref{limit path}. 
In this case the $p_i$ are secret $4$-IETs at level $i$ contained in
$M(S,i)\Delta$ or $M(E,i)\Delta$. The paths $P_i$ connecting $p_i$ and $p_{i+1}$ are contained in $M(S,m_i)$ or $M(E,q_i)$,
where the $m_i$ or $q_i$ is given 
by Lemma \ref{lem:well connected}. The $m_i$ and $q_i$ go to infinity with $i$. Since $S$ and $E$ are uniquely 
ergodic and have Rauzy induction defined for all $k$ we have that $M(S,i)\Delta$ and $M(E,i)\Delta$ contract to $S$ and $E$ respectively. This verifies Lemma~\ref{limit path}.

Otherwise, assume that $S$ does not have all powers of Rauzy induction
defined.  We may without loss of generality assume that $E$
does. We prove the theorem by assuming that we have inductively proved Theorem \ref{thm:gen case} for all $4\leq m<n$.
 There exists an interval $J$ so that the induced map $S|_J$ is a uniquely 
ergodic IETs on $k<d$ symbols. There is a $d \times k$ matrix $M$ so
that $L(S)=ML(S|_J)$ where in an abuse of notation $L(S|_J)$ is a
vector with $k$ entries. Because every irreducible permutation has a pair
where it acts irreducibly, we have a secret $4$-IET $S'$ of the form
$M\bar{v}$ for some $v$ whose non-zero entries are a subset of
$L(S|_J)$. By induction, the $k$-IET $S|_J$ can be joined by a path of
uniquely ergodic IETs to $S'$. Applying $M$ to this path connecting $S|_J$ to $S'$ we obtain 
a path connecting a $S$ to a secret 4-IET (with length vector given by $ML(S')$). 
By the first case, $E$ is path connected by
uniquely ergodic IETs to a secret $4$-IET of some depth and this secret
$4$-IET is path connected by uniquely ergodic IETs to every uniquely
ergodic secret $4$-IET. We have just shown that one of these is path connected  $S$ by a path of uniquely ergodic IETs.
\end{proof}

\appendix
\section{Composite paths}
\label{sec:composite-paths}
In Figure~\ref{fig:diagram-a}
\begin{description}
\item[$\alpha1$] $(4132)\to(3142)\to(3142)\to(4132)$
\item[$\alpha2$] $(4132)\to(4213)\to(4321)\to(2431)\to(3241)\to(4321)$
\item[$\alpha3$] $(2431)\to(3241)\to(3241)\to(4321)$
\item[$\alpha4$] $(2431)\to(2413)\to(2431)$
\item[$\alpha5$] $(4132)\to(4213)\to(4321)$
\end{description}
In Figure~\ref{fig:diagram-b}
\begin{description}
\item[$\beta1$] $(2431)\to(2413)\to(2413)\to(2431)$
\item[$\beta2$] $(2431)\to(3241)\to(4321)\to(4132)\to(4213)\to(4321)$
\item[$\beta3$] $(4132)\to(4213)\to(4213)\to(4321)$
\item[$\beta4$] $(4132)\to(3142)\to(4132)$
\item[$\beta5$] $(2431)\to(3241)\to(4321)$
\end{description}
In Figure~\ref{fig:diagram-cd}
\begin{description}
\item[$\gamma1$] $(2431)\to(3241)\to(3241)\to(4321)$
\item[$\gamma2$] $(2431)\to(2413)\to(2431)$
\item[$\gamma3$] $(4132)\to(3142)\to(3142)\to(4132)$
\item[$\gamma4$] $(4132)\to(4213)\to(4321)$
\item[$\gamma5$] $(2431)\to(3241)\to(4321)$
\item[$\gamma6$] $(2431)\to(2413)\to(2413)\to(2431)$
\item[$\gamma7$] $(4132)\to(3142)\to(4132)$
\item[$\gamma8$] $(4132)\to(4213)\to(4213)\to(4321)$
\end{description}
In Figure~\ref{fig:diagram-efgh}
\begin{description}
\item[$\xi1$] $(4213)\to(4213)\to(4321)$
\item[$\xi2$] $(2413)\to(2413)\to(2431)\to(2413)\to(2413)\to(2431)$
\item[$\xi3$] $(2431)\to(3241)\to(3241)\to(4321)\to(4132)\to(4213)\to(4321)$
\item[$\xi4$] $(2431)\to(2413)\to(2413)\to(2431)$
\item[$\xi5$] $(3241)\to(3241)\to(4321)$
\item[$\xi6$] $(3142)\to(3142)\to(4132)\to(3142)\to(3142)\to(4132)$
\item[$\xi7$] $(4132)\to(4213)\to(4213)\to(4321)\to(2431)\to(3241)\to(4321)$
\item[$\xi8$] $(4132)\to(3142)\to(3142)\to(4132)$
\end{description}


\begin{thebibliography}{xxxxx}
\bibitem[AC]{athreya-chaika}  Athreya, J. and Chaika, J.: \textit{The Hausdorff dimension of Non-Uniquely Ergodic directions in $\mathcal{H}(2)$ is almost everywhere $1/2$}. to appear in Geom. Topol. 
\bibitem[G09]{gab}  Gabai, D: \textit{Almost filling laminations and the connectivity of ending lamination space}. Geom. Topol. 13 (2009), no. 2, 1017--1041
\bibitem[G11]{gabai2}  Gabai, D: \textit{On the topology of ending lamination
  space.} arXiv preprint arXiv:1105.3648 (2011).
\bibitem[HP11]{henspryz}  Hensel, S; Przytycki, P: \textit{The ending lamination space of the five-punctured sphere is the N\"{o}beling curve}. J. Lond. Math. Soc. (2) 84 (2011), no. 1, 103--119.
\bibitem[KS67]{KS} Katok, A. B; Stepin, A. M: \textit{Approximations in ergodic theory.} Uspehi Mat. Nauk 22 1967 no. 5 (137), 81--106. 
\bibitem[K75]{keane}  Keane, M: \textit{Interval exchange transformations}. Math. Z. 141 (1975), 25--31. 
\bibitem[Ker85]{ker}Kerckhoff, S. P: \textit{Simplicial systems for interval exchange maps and measured foliations.} Ergodic Theory Dynam. Systems 5 (1985), no. 2, 257--271.
\bibitem[LS09]{leinshleim} Leininger, C; Schleimer, S: \textit{Connectivity of the space of ending laminations.} Duke Math. J. 150 (2009), no. 3, 533--575.
\bibitem[LS11]{leinshleim2} Leininger, C; Schleimer, S: \textit{Hyperbolic
    spaces in Teichm\"uller spaces.} arXiv preprint arXiv:1110.6526 (2011).
\bibitem[M82]{masur}  Masur, H: \textit{Interval exchange transformations and measured foliations.} Ann. of Math. (2) 115 (1982), no. 1, 169--200.
\bibitem[MS91]{masur smillie} Masur, H; Smillie, J: \textit{Hausdorff Dimension of Sets of Nonergodic Measured Foliations}, Ann. of Math. (2) 134, No. 3 (1991), pp. 455-543
\bibitem[V06]{viana survey}Viana, M: \textit{Ergodic theory of interval exchange maps.} Rev. Mat. Complut. 19 (2006), no. 1, 7--100.
\bibitem[V78]{iet v}  Veech, W:  \textit{Interval exchange transformations.} J. Analyse Math. 33 (1978), 222--272. 
\bibitem[V82]{gauss measures}Veech, W: \textit{Gauss measures for transformations on the space of interval exchange maps.} Ann. of Math. (2) 115 (1982), no. 1, 201--242.
\bibitem[Y10]{yoccoz} Yoccoz, J-C: Interval exchange maps and translation surfaces. Homogeneous flows, moduli spaces and arithmetic, 1ऱ69, Clay Math. Proc., 10, Amer. Math. Soc., Providence, RI, 2010.
\bibitem[Z06]{zorich}Zorich, A. Flat Surfaces. \textit{Frontiers in Number
  Theory, Physics and Geometry I} (2006): 439--585.
\end{thebibliography}
\end{document}